\documentclass[11pt,a4paper]{article}

\pdfoutput=1

\usepackage[utf8]{inputenc}
\usepackage[T1]{fontenc}
\usepackage{babel}
\usepackage{amsthm, amssymb, amsmath, amsfonts, mathrsfs}
\usepackage{mathtools}
\usepackage[colorlinks=true, pdfstartview=FitV, linkcolor=blue, citecolor=blue, urlcolor=blue,pagebackref=false]{hyperref}

\usepackage{microtype}
\usepackage{tikz}

\usepackage{geometry}
\usepackage{fancyhdr}
\usepackage{setspace}
\usepackage[nottoc]{tocbibind}

\newcommand{\keywords}[1]
{\noindent
  {\small	
  \textbf{\text{Keywords: }} #1}
}

\newcommand{\subjclass}[2][1991]
{\noindent
  {\small	
  \textbf{\text{MSC 2010 Subject Classification: }} #2}
}

\newtheorem{thm}{Theorem}[section]
\newtheorem{prop}[thm]{Proposition}
\newtheorem{lem}[thm]{Lemma}

\newtheorem{assumption}[thm]{Assumption}
\theoremstyle{remark}
\newtheorem{rmk}[thm]{Remark}

\newcommand{\no}{\nonumber}
\newcommand{\lf}{\left}
\newcommand{\rh}{\right}
\newcommand{\dr}{\partial}
\newcommand{\gr}{\nabla}
\newcommand{\eps}{\varepsilon}
\newcommand{\D}{\Delta}
\newcommand{\e}{\eta}
\newcommand{\Xb}{{\overline{X}}}
\newcommand{\Yb}{{\overline{Y}}}
\newcommand{\Zb}{{\overline{Z}}}
\newcommand{\Cb}{\overline{C}}
\newcommand{\pb}{{\overline{p}}}
\newcommand{\mub}{{\overline{\mu}}}
\renewcommand{\th}{\tilde{h}}
\newcommand{\tgr}{\tilde{\nabla}}
\newcommand{\tR}{\tilde{R}}
\newcommand{\tb}{\tilde{\beta}}
\newcommand{\E}{\mathbb{E}}
\renewcommand{\P}{\mathbb{P}}
\newcommand{\R}{\mathbb{R}}

\renewcommand{\O}{{\mathcal{O}}}
\newcommand{\CP}{\mathcal{P}}
\newcommand{\M}{\mathcal{M}}
\newcommand{\F}{{\mathcal{F}}}
\newcommand{\CL}{{\mathcal{L}}}
\newcommand{\C}{\mathscr{C}}
\renewcommand{\L}{\mathscr{L}}
\renewcommand{\d}{{\mathrm{d}}}
\newcommand{\KL}{{\mathrm{KL}}}
\newcommand{\TV}{{\mathrm{TV}}}
\newcommand{\tr}{{\mathrm{tr}}}
\newcommand{\T}{{\mathrm{T}}}

\newcommand{\rank}{{\mathrm{rank}}}

\numberwithin{equation}{section}

\title{Convergence of Simulated Annealing Using Kinetic Langevin Dynamics}

\author{
Xuedong He\thanks{Department of Systems Engineering and Engineering Management, The Chinese University of Hong Kong.
\href{mailto:xdhe@se.cuhk.edu.hk}{xdhe@se.cuhk.edu.hk}
}
\and Xiaolu Tan\thanks{Department of Mathematics, The Chinese University of Hong Kong.
\href{mailto:xiaolu.tan@cuhk.edu.hk.}{xiaolu.tan@cuhk.edu.hk.}
}
\and Ruocheng Wu\thanks{Department of Systems Engineering and Engineering Management, The Chinese University of Hong Kong.
\href{mailto:rcwu@se.cuhk.edu.hk.}{rcwu@se.cuhk.edu.hk.}
}
}

\date{}

\begin{document}

\maketitle
\thispagestyle{fancy}

\begin{abstract}

	We study the simulated annealing algorithm based on the kinetic Langevin dynamics, in order to find the global minimum of a non-convex potential function.
	For both the continuous time formulation and a discrete time analogue, we obtain the convergence rate results under technical conditions on the potential function,
	together with an appropriate choice of the cooling schedule and the time discretization parameters.

\end{abstract}


%


%
%
%
%
%
%
%
%


\section{Introduction}
	
	Simulated annealing has always been an important method to find the global minimum of a given function $U: \R^d  \longrightarrow \R$, especially when $U$ is non-convex.
	Classical studies on the simulated annealing have been mainly focused on an algorithm based on the overdamped Langevin dynamic:
	\begin{equation} \label{eq:ol}
		\d X_t = -\gr U(X_t)\d t + \sqrt{2\eps_t} \d B_t,
	\end{equation}
	where $(B_t)_{t \ge 0}$ is a standard $d$-dimensional Brownian motion and $(\eps_t)_{t \ge 0}$ is a temperature parameter that turns to $0$ as $t \longrightarrow \infty$.
	Notice that, with constant temperature $\eps_t \equiv \eps$ and under mild conditions on $U$,
	the process $X$ in \eqref{eq:ol} is the standard overdamped Langevin dynamic and has the invariant measure $\mu^*_{o, \eps}(\d x) \propto \exp\left(-U(x)/\eps \right) \d x $.
	With small $\eps > 0$, samples from $\mu^*_{o,\eps}$ would approximately concentrate around the global minimum of function $U$,
	which is the intuition of the simulated annealing algorithm.

	Since the introduction of the simulated annealing algorithm by Kirkpatrick, Gellatt and Vecchi \cite{KirkpatrickGelattVecchi}, many works have been devoted to the convergence analysis of \eqref{eq:ol}; see e.g., Geman and Hwang \cite{GemanHwang86}, Chiang, Hwang and Sheu \cite{ChiangHwangSheu87}, Royer \cite{RoyerSICON89}, Holley, Kusuoka and Stroock \cite{HolleyStroockJFA89}, Miclo \cite{Miclo92}, Zitt \cite{ZittSPA08}, Fournier and Tardif \cite{FournierTardifJFA21}, Tang and Zhou \cite{TangZhou21}, etc.
	It has been shown that, the cooling schedule $\eps_t$ should be at least of the order $\frac{E}{\log t}$ as $t \longrightarrow \infty$ for some constant $E > 0$, in order to ensure the convergence of $X_t$ to the global minimum of $U$ as $t\longrightarrow \infty$.
	Intuitively, this cooling schedule allows the diffusion process to have enough time to escape from the local minima and at the same time to explore the whole space in order to find the global minimum of $U$; finally, the annealing process will ``freeze'' at the global minimum of $U$ as $\eps_t \longrightarrow 0$.
	We would like to mention in particular the recent paper by Tang and Zhou \cite{TangZhou21}, where the authors derived a convergence rate result of \eqref{eq:ol},
	where a fine estimation of the log-Sobolev inequality in Menz and Schlichting \cite{MenzSchlichtingAOP14} for invariant measure $\mu^*_{\eps,o}$ with low-temperature (small $\eps > 0$) has been crucially used.
	Moreover, they have also analyzed a corresponding discrete time scheme of \eqref{eq:ol} and obtained a convergence rate result.
	Notice that in practice it is the discrete time scheme which is implemented to find the optimizer of $U$.

	Motivated by the above works, we will study in this paper the simulated annealing based on the kinetic (underdamped) Langevin dynamic,
	that is, the process $(X,Y) = (X_t, Y_t)_{t \ge 0}$ defined by
	\begin{equation} \label{eq:kl}
		X_t = X_0 + \!\int_0^t\! Y_s \d s,
		\;\;\;\;
		Y_t = Y_0 - \int_0^t\! \gr_x U(X_s)\d s - \int_0^t \! \theta Y_s\d s + \int_0^t \!\sqrt{2\eps_s} \d B_s,
	\end{equation}
	where $\theta > 0$ is a fixed constant and $(\eps_t)_{t \ge 0}$ is a cooling schedule satisfying $\eps_t \longrightarrow 0$ as $t \longrightarrow \infty$.
	Moreover, we will also study a discrete time scheme of \eqref{eq:kl}.
	More precisely, consider a sequence $(\D t_k)_{k \ge 0}$ of time steps and define the discrete time grid $0 = T_0 < T_1 < \cdots $ by
	$$
		T_k
		~:=~
		\sum_{j=0}^{k-1} \D t_j .
	$$
	The discrete time scheme will be defined on the grid $(T_k)_{ k \ge 0}$.
	For ease of presentation and the convergence analysis later, we will write this scheme as a continuous time process $(\Xb, \Yb) = (\Xb_t, \Yb_t)_{t \ge 0}$ by using the time freezing function $\e(t) := \sum_{k=0}^{\infty} T_k \mathbf{1}_{\{t\in[T_k, T_{k+1})\}}$.
	Then, the discrete time scheme process $(\Xb, \Yb) $ is defined by
	\begin{equation} \label{eq:diskl}
		\overline{X}_t = X_0 + \int_0^t\! \overline{Y}_s \d s,
		\;\;\;\;
		\overline{Y}_t = Y_0 + \int_0^t \!\! \lf( -\gr_x U(\Xb_{\e(s)}) - \theta \Yb_s \rh) \d s + \int_0^t \!\! \sqrt{2\eps_{\e(s)}} \d B_s  .
	\end{equation}
	Notice that the above scheme is the second-order scheme, rather than the Euler scheme of \eqref{eq:kl}.
	This scheme can be explicitly re-written on the time grid $(T_k)_{k \ge 0}$ and hence is implementable (see \eqref{eq:xbtybt} and \eqref{eq:implement} for details). This second-order scheme has also been introduced and studied for  standard underdamped Langevin dynamic, i.e., for \eqref{eq:kl} with constant temperature $\eps_t \equiv \eps_0$; see, e.g., Cheng, Chatterji, Bartlett and Jordan \cite{ChengCOLT18}, Zou, Xu and Gu \cite{ZouXuGuNIPS19}, Gao, G{\"u}rb{\"u}zbalaban and Zhu \cite{GaoZhuOR21} and Ma, Chatterji, Cheng, Flammarion, Bartlett and Jordan \cite{MaBEJ21}.

	For the kinetic simulated annealing process $(X,Y)$ in \eqref{eq:kl}, a convergence result without convergence rate has already been established by Journel and Monmarch\'e \cite{JournelMonmarche21}.
	In the present paper, we aim at obtaining some convergence rate results for both the simulated annealing process $(X,Y)$ in \eqref{eq:kl} and the discrete time scheme $(\Xb, \Yb)$ in \eqref{eq:diskl}.
	To the best of our knowledge, we are the first to study the convergence of the kinetic simulated annealing algorithm in the discrete time framework.
	Let us also mention that, by cooling the parameter $\theta$ instead of $\eps$ in \eqref{eq:kl}, 
	Monmarch\'e \cite{MonmarchePTRF18} studied an alternative kinetic simulated annealing process and derived a convergence rate for it (see Remark \ref{rem:compare_monmarche} in the following for a detailed comparison of his work and ours).

	The remainder of the paper is organized as follows.
	We first introduce some notations.
	In Section~\ref{s:main}, we state the assumptions and our main results,
	and provide the main idea of the proofs, together with some discussions on the related literature.
	The proof of the convergence rate of \eqref{eq:kl} is given in Section~\ref{s:conSA}, and the convergence rate of \eqref{eq:diskl} is given in Section~\ref{s:disSA}.

	\noindent {\bf Notations.} $\mathrm{(i)}$ Denote by $C^{\infty}(\R^d)$, or simply $C^{\infty}$, the collection of all smooth (i.e., infinitely differentiable) functions $f: \R^d \longrightarrow \R$.
	For $f \in C^{\infty}$, let $\gr f, \gr_x^2 f,$ and $\Delta f $ denote the gradient, Hessian, and Laplacian of $f$, respectively.
	For a smooth vector field $v:\R^d\longrightarrow\R^d$, $\gr\cdot v$ denotes the divergence of $v$.
	For vectors $a,b\in\R^d$, $\langle a,b \rangle$ is their inner product and $|a|=\sqrt{\langle a,a\rangle}$ is the Euclidean norm of $a$.
	For two matrices $M,N\in\M_{d\times d}(\R) $ their Frobenius inner product is defined as $\langle M, N \rangle_F=\tr (M^{\T} N)=\sum_{i,j=1}^d M_{ij}N_{ij} $,
	and $\| M\|_F =\sqrt{\langle M ,M\rangle_F} $ is the Frobenius norm of $M$.
	A function $\eps \longmapsto w(\eps)$ is said to be sub-exponential if $\eps \log w(\eps)\longrightarrow 0$ as $\eps\longrightarrow 0$.

	\noindent $\mathrm{(ii)}$  We denote by $\CP(\R^d)$ (resp. $\CP(\R^d \times \R^d)$) the collection of all probability measures on $\R^d$ (resp. $\R^d \times \R^d$).
	For functions $f$ and $g$ defined on $\R_+$,
	the symbol $f=\O(g)$ means that $f/g$ is bounded when some problem parameter tends to $0$ or $\infty$.

\section{Main Results and Literature}
\label{s:main}

	We will state our main convergence rate results 
	and then discuss the main idea of proof as well as some related literature.

\subsection{Main Results}

	We first provide some conditions on the (potential) function $U: \R^d \longrightarrow \R$.
	Without loss of generality, we assume that $\min_{x\in \R^d} U(x)=0$ throughout the paper.

	\begin{assumption} \label{assumption1}
		$\mathrm{(i)}$ The function $U \in C^{\infty}(\R^d)$ and all its derivatives have at most polynomial growth.
		The gradient $\gr U$ is $L$-Lipschitz for constant $L > 0$. Moreover, $U$ is $(r,m)$-dissipative in the sense that for some positive constants $r > 0$ and $m > 0$,
		$$
			\gr U(x)\cdot x\ge r|x|^2-m,
			~~\mbox{for all}~
			x \in \R^d.
		$$
		
		\noindent $\mathrm{(ii)}$ The function $U$ has a finite number of local minimizers and $\gr^2 U$ is non-degenerate at the local minimizers.
		Moreover, $U$ has at least one non-global minimizer.
	\end{assumption}

	\begin{rmk}
		$\mathrm{(i)}$ In the literature of the standard kinetic (underdamped) Langevin dynamic (i.e., \eqref{eq:kl} with constant temperature $\eps_t \equiv \eps_0$), 
		the dissipative condition on $U$ is a standard Lyapunov condition to ensure the ergodicity of the process; see e.g., Eberle, Guillin and Zimmer \cite{EberleGuillinZimmerAOP19} and Mattingly, Stuart and Higham \cite{MattinglyStuartHigham02}.
		The Lipschitz condition on $\gr U$ is also usually imposed to obtain quantitative exponential convergence rate of the law of standard kinetic Langevin dynamic to its invariant measure; see, e.g., \cite{EberleGuillinZimmerAOP19, GaoZhuOR21, MaBEJ21}.
		In particular, this condition ensures that the process $(X,Y)$ in \eqref{eq:kl} is well defined.

		\noindent $\mathrm{(ii)}$ The Lipschitz condition on $\gr U$, together with the dissipative condition, implies that there exists a positive constant $K$ such that
		\begin{equation}\label{quadratic}
			\frac{r}{3}|x|^2-K \le U(x) \le L|x|^2+K,
			~\mbox{for all}~
			x \in \R^d.
		\end{equation}
		For a proof, see, e.g., Raginsky, Rakhlin and Telgarsky \cite[Lemma 2]{RaginskyRakhlinCOLT17}.
	\end{rmk}

	Let $m_U$ and $M_U$ denote respectively the set of local minima and the set of global minima of $U$.
	We then define a constant $E_* > 0$, the so-called critical depth of $U$, by
	$$
		E_*
		:=
		\max_{ \begin{subarray}{c} x\in m_U\\ y\in M_U \end{subarray} } \inf \left\{ \max_{s\in[0,1]} U(\gamma(s))-U(x) ~:\gamma\in\C\left( [0,1], \R^d  \right),\ \gamma(0)=x, \ \gamma(1)=y \right\}.
	$$

	\begin{assumption} \label{assumption2}
		$\mathrm{(i)}$
		The initial distribution of $(X_0,Y_0)$, denoted as $\mu_0= \L (X_0, Y_0)$, has a $C^{\infty}$ density function $p_0$.
		Moreover, the initial Fisher information $\int_{\R^d \times \R^d} \frac{|\gr p_0(x,y)|^2}{p_0(x,y)}\ \d x\d y $ is finite and
		$ \E[|X_0|^m+|Y_0|^m] <\infty$ for each $m\ge 1$.

		\noindent $\mathrm{(ii)}$
		The function $t \longmapsto \eps_t$ is positive, non-increasing and differentiable.
		Moreover, for some time $t_0$ and a constant $E > E_*$, one has $\eps_t=\frac{E}{\log t}$ for all $t>t_0$.
	\end{assumption}

	\begin{rmk}
		To obtain the convergence of the simulated annealing using overdamped Langevin dynamic in \eqref{eq:ol},
		it is standard to consider the cooling schedule {\color{blue}$\eps_t = \O \lf( \frac{1}{\log t}\rh)$ as $t \longrightarrow \infty$};
		see, e.g., \cite{GemanHwang86, ChiangHwangSheu87, RoyerSICON89, HolleyStroockJFA89, Miclo92, TangZhou21}. 
		For the kinetic simulated annealing process \eqref{eq:kl}, this cooling schedule is also assumed in Journel and Monmarch\'e \cite{JournelMonmarche21} to obtain a convergence result without convergence rate.
		Moreover, it is proved in \cite{JournelMonmarche21} that the convergence may fail for a faster cooling schedule.
		In Monmarch\'e \cite{MonmarchePTRF18}, for an alternative kinetic simulated annealing process with cooling schedule on parameter $\theta$, a similar cooling schedule is also assumed.
		Moreover, our technical conditions on $U$ in Assumption \ref{assumption1} are also generally motivated by those in \cite{MonmarchePTRF18}.
	\end{rmk}

	Let us now provide a first convergence rate result on $(X, Y)$ defined by \eqref{eq:kl}.
	By a time change argument, we can easily reduce \eqref{eq:kl} to the case with $\theta =1$.
	For this reason, we will always assume $\theta = 1$ in the rest of paper.
	Also, recall that the condition $\min_{x \in \R^d} U(x) = 0$ is assumed throughout the paper.

	\begin{thm} \label{thm1}
		Let Assumptions~\ref{assumption1} and \ref{assumption2} hold true.
		Then, for any constants $\delta > 0$ and $\alpha > 0$, there exists some constant $C > 0$ such that
		\begin{equation*}
			\P \lf(U (X_t) > \delta\rh)
			~\le~
			C t^{-\min \lf( \frac{\delta}{E} , \frac{1}{2} \lf( 1-\frac{E_*}{E} \rh)  \rh) + \alpha }, \quad \mbox{for all}\;\; t>0.
		\end{equation*}
	\end{thm}

	\begin{rmk} \label{rmk1}

		$\mathrm{(i)}$ In \cite{JournelMonmarche21}, the authors used localization arguments to obtain a convergence result for \eqref{eq:kl} with conditions on the potential function $U$ weaker than ours. They, however, did not derive the convergence rate.

		\noindent $\mathrm{(ii)}$ The convergence rate in Theorem~\ref{thm1} is the same to that in Miclo \cite{Miclo92} and Tang and Zhou \cite{TangZhou21} for the simulated annealing using overdamped Langevin dynamic,
		and also to that in Chak, Kantas and Pavliotis \cite{PavliotisGeneralisedLangevin} for the simulated annealing using generalized Langevin process. While higher order Langevin dynamics are often used in MCMC as accelerated versions compare to the overdamped Langevin dynamic (see for instance \cite{MaBEJ21, GaoZhuOR21, MouJMLR21highorderLangevin}), this is not the case in the simulated annealing problem.
		Indeed, it has been observed that the convergence behavior of the annealed process will be mainly determined by the potential function $U$,
		but not the process used; see for instance the discussion in \cite[Remarks after Theorem 1]{MonmarchePTRF18}.
	\end{rmk}

	We now present the convergence rate result for the discrete time scheme $(\Xb, \Yb)$ as defined in \eqref{eq:diskl}.

	\begin{thm} \label{thm2}
		Let Assumptions~\ref{assumption1} and \ref{assumption2} hold true.
		Assume in addition that $\gr_x^2 U$  is $L'$-Lipschitz, i.e. $\|\gr_x^2 U(x) -\gr_x^2 U(y)\|_F \le L'|x-y| $ for all $x,y \in\R^d$, for some constant $L'>0$.
		And that the time step size parameters $(\Delta t_k)_{k\ge 0}$ satisfies
		\begin{equation} \label{eq:Cond_Tk}
			\lim_{k \to \infty} T_k = \infty
			\;\;\mbox{and}\;\;
			\limsup_{k \to \infty}  \D t_k \sqrt{T_k}  < \infty.
		\end{equation}
		Then for all constants $\delta > 0$ and $\alpha > 0$, there exists a constant $C> 0$, such that
		\begin{equation*}
			\P \lf(U (\Xb_{T_k}) > \delta\rh) \le C {T_k}^{-\min \lf( \frac{\delta}{E} , \frac{1}{2} \lf( 1-\frac{E_*}{E} \rh)  \rh) + \alpha }, \quad \mbox{for all}\;\; k\ge 1.
		\end{equation*}
	\end{thm}

	\begin{rmk}
		$\mathrm{(i)}$ The additional Lipschitz condition on $\gr_x^2 U$ will be essentially used to control the discretization error in the discrete time scheme \eqref{eq:diskl}.

		\noindent $\mathrm{(ii)}$ We need that $\Delta t_k \longrightarrow 0$ as $T_k \longrightarrow \infty$  to control the (cumulative) discretization error in the scheme \eqref{eq:diskl}.
		At the same time, $\Delta t_k$ should not decrease too fast, so that $T_k \longrightarrow \infty$ as $k \longrightarrow \infty$ and thus $(X_{T_k})_{k \ge 0}$ can reach the global minima of $U$. This explains the condition \eqref{eq:Cond_Tk}.

		\noindent $\mathrm{(iii)}$ Let $C_1 > 0$, $C_2 > 0$ be two constants, let us define $(\D t_k)_{k \ge 0}$ by $\D t_0 := C_1$ and $\D t_k := C_2 T_k^{-1/2} $ for $k \ge 1$.
		Then, it is straightforward to check that
		$$
			T_{k+1}^2
			~=~
			 (T_k + \D t_k)^2 = T^2_k + C^2/T_k + 2 CT_k^{1/2}
			~\ge~
			T^2_k + 3C^{4/3}.
		$$
		Therefore, $T_k \longrightarrow \infty$ as $k \longrightarrow\infty$, and thus condition \eqref{eq:Cond_Tk} holds.

		\noindent $\mathrm{(iv)}$ In Tang and Zhou \cite{TangZhou21}, for the discrete simulated annealing based on the overdamped Langevin dynamic \eqref{eq:ol},
		the authors required the step size $\D t_k$ satisfying
		\begin{equation*}
			\lim_{k \to \infty} T_k = \infty
			\;\;\mbox{and}\;\;
			\limsup_{k \to \infty}  \D t_k T_k  < \infty,
		\end{equation*}
		which is a little stronger than our condition \eqref{eq:Cond_Tk}.
		Of course, \eqref{eq:Cond_Tk} is only a sufficient condition to ensure the convergence rate result.
		
	\end{rmk}

	Notice that  \eqref{eq:diskl} is a linear SDE on each time interval $[T_k,T_{k+1}]$, so we can solve it explicitly.
	Namely, given the value $(\Xb_{T_k}, \Yb_{T_k})$, we have for all $t \in [T_k, T_{k+1}]$,
	\begin{equation} \label{eq:xbtybt}
	\begin{cases}
		\Xb_t  = \Xb_{T_k} + \lf( 1-e^{-(t-T_k)} \rh) \Yb_{T_k} - \lf(  t-T_k - \lf( 1-e^{-(t-T_k)} \rh) \rh) \gr_x U(\Xb_{T_k}) + D_x(t), \\
		\Yb_t  = e^{-(t-T_k)} \Yb_{T_k} - \lf( 1-e^{-(t-\T_k)} \rh) \gr_x U(\Xb_{T_k}) + D_y(t),
	\end{cases}
	\end{equation}
	with
	\begin{equation*}
		D_x (t)  = \int_{T_k}^t D_y(s)\d s ,
		\;\;\;\;
		D_y (t)  = \sqrt{2\eps_{T_k}} \int_{T_k}^t e^{-(t-s)}\d B_s,
	\end{equation*}
	where $B_t$ is the Brownian motion in \eqref{eq:diskl}.
	Therefore, we can implement $(\Xb, \Yb)$ on the discrete time grid $(T_k)_{k \ge 0}$ in an exact way.
	More precisely, by abbreviating $(\Xb_{T_k}, \Yb_{T_k}, \eps_{T_k})$ to $(\Xb_k, \Yb_k, \eps_k)$,
	we have
	\begin{equation} \label{eq:implement}
	\begin{cases}
		\Xb_{k+1}  = \Xb_k + \lf( 1-e^{-\D t_k} \rh) \Yb_k - \lf(  \D t_k - \lf( 1-e^{-\D t_k} \rh) \rh) \gr_x U(\Xb_k) + D_x(k), \\
		\Yb_{k+1}  = e^{-\D t_k} \Yb_k - \lf( 1-e^{-\D t_k} \rh) \gr_x U(\Xb_k) + D_y(k),
	\end{cases}
	\end{equation}
	where $\lf(D_x(k), D_y(k) \rh)$ is Gaussian vector in $\R^d \times\R^d $ independent of $(\Xb_k ,\Yb_k)$ with mean zero and covariance matrix
	\begin{equation} \label{eq:sigmak}
	\Sigma_k=
	\begin{pmatrix}
		\Sigma_{11}(k) I_d & \Sigma_{12}(k) I_d\\
		\Sigma_{12}(k) I_d & \Sigma_{22}(k) I_d
	\end{pmatrix},
	\end{equation}
	with $\Sigma_{11}(k) := \eps_k\lf( 2\D t_k -3 + 4e^{-\D t_k} - e^{-2\D t_k} \rh)$, $\Sigma_{12}(k) :=\eps_k \lf( 1 - 2e^{-\D t_k} + e^{-2\D t_k} \rh)$ and $\Sigma_{22}(k) :=  \eps_k \lf( 1 - e^{-2\D t_k} \rh) $.

\subsection{Main Idea of Proofs and Related Literature}
\label{s:proofsketch}

	Recall that for two probability measures $\mu, \nu\in\CP(\R^d)$ (or in $\CP(\R^d \times \R^d)$) such that $\mu\ll\nu$,
	the relative entropy $\KL(\mu|\nu)$ and the Fisher information $I(\mu|\nu)$ are defined by
	$$
		\KL(\mu|\nu) := \int\log \left(\frac{\d \mu}{\d \nu} \right)\d \mu,
		\;\;\;
		I(\mu|\nu):=\int\left|\gr\log\frac{\d \mu}{\d\nu} \right|^2\d\mu.
	$$

	For the simulated annealing process \eqref{eq:ol} using overdamped Langevin dynamic, let us denote $\mu_{o,t} := \L(X_t)$ and $\mu^*_{o,\eps} \propto\exp(-U(x)/\eps)\d x$ the invariant measure of standard overdamped Langevin dynamic with constant temperature $\eps$.
	To deduce their convergence rate result, Tang and Zhou \cite{TangZhou21} analyze the evolution of $\KL(\mu_{o,t}| \mu^*_{\eps_t,o})$.
	A key step consists in obtaining
	\begin{equation} \label{eq:olexp}
		\frac{\d}{\d t} \KL(\mu_{o,t}| \nu) \big|_{\nu = \mu^*_{o,\eps_t}}
		~=~
		-I(\mu_{o,t}|\mu^*_{o,\eps_t})
		~\le~
		- \rho(\eps_t) \ \KL(\mu_{o,t}|\mu^*_{o,\eps_t})
	\end{equation}
	for a good constant $\rho(\eps_t)$ (depending on $\eps_t$).
	The equality in \eqref{eq:olexp} follows the so-called de Bruijn’s identity, which is the entropy dissipation equation for the standard overdamped Langevin dynamic; see for instance Chapter 5.2. in the monograph Bakry, Ledoux and Gentil \cite{BGL} or Otto and Villani \cite{OttoVillaniJFA00} for a more general setting.
	The inequality in \eqref{eq:olexp} follows from the log-Sobolev inequality (LSI).
	This step is also the classical way to deduce the exponential ergodicity of the standard overdamped Langevin dynamic; see for instance Arnold, Markowich, Toscani and Unterreiter\cite{ArnoldMarkowichCPDE01} and Pavliotis \cite{PavliotisLangevinBook}.

	For the kinetic simulated annealing process \eqref{eq:kl} (with $\theta=1$), let us denote $\mu_t := \L (X_t, Y_t)$ and by
	$$
		\mu^*_{\eps}(\d x\,\d y) \propto\exp\left(-\frac{1}{\eps} \left( U(x)+\frac{|y|^2}{2} \right) \right) \d x\, \d y
	$$
	the invariant measure of standard kinetic Langevin dynamic with constant temperature $\eps$ (i.e. \eqref{eq:kl} with $\eps_t \equiv \eps$).
	It is, however, no longer helpful to use the relative entropy $\KL(\mu_t | \mu^*_{\eps_t})$ for the convergence analysis.
	Indeed, because the Brownian motion only appears in the $y$-direction in \eqref{eq:kl}, the time derivative of the relative entropy only gives a part of the Fisher information:
	$$
		\frac{\d}{\d t}\KL(\mu_t |  \nu) \big|_{\nu = \mu^*_{\eps}} = -\eps\int \left| \gr_y\log\frac{\d \mu_t}{\d \mu^*_{\eps}}  \right|^2\d \mu_t \not=-\eps I(\mu_t|\mu^*_{\eps}).
	$$
	Thus, we can not use the LSI to proceed as in \cite{TangZhou21} for the overdamped simulated annealing.

	This is also the main reason why we cannot use the relative entropy to deduce the exponential ergodicity of the standard kinetic Langevin dynamic,
	for which a successful alternative method is the so-called hypocoercivity method, initiated in Desvillettes and Villani \cite{DesvillettesVillani01}, H{\'e}rau \cite{Herau06, HerauJFA07} (see also Villani \cite{VillaniHypocoercivityAMS} for a detailed presentation).
	As in Monmarch\'e \cite{MonmarchePTRF18}, we will adopt the hypocoercivity method to study our kinetic simulated annealing processes in \eqref{eq:kl} and \eqref{eq:diskl}.
	More precisely, for two probability measures $\mu,\nu\in\CP(\R^d \times \R^d)$, we consider the distorted relative entropy $H_{\gamma}(\mu| \nu)$ defined by
	\begin{equation}\label{disent}
		H_{\gamma}(\mu|\nu):=\int \left( \lf|\gr_x\log\frac{\d \mu}{\d \nu} + \gr_y\log\frac{\d \mu}{\d \nu}\rh|^2 +\gamma \log \frac{\d \mu}{\d \nu} \right) \d \mu,
	\end{equation}
	where $\gamma>0$ is a constant.
	To deduce the convergence rate result of \eqref{eq:kl}, we will choose $\gamma$ to be a function of $\eps_t$, and then analyze the evolution of $H_{\gamma(\eps_t)}(\mu_t|\mu^*_{\eps_t})$:
	\begin{equation} \label{eq:der}
		\frac{\d}{\d t} H_{\gamma(\eps_t)}(\mu_t|\mu^*_{\eps_t}) = \dr_{\mu,t} H_{\gamma(\eps_t)} (\mu_t|\mu^*_{\eps_t}) +
		\dr_{\eps, t} H_{\gamma(\eps_t)} (\mu_t|\mu^*_{\eps_t}),
	\end{equation}
	where
	\begin{equation} \label{eq:der1}
		\lf. \dr_{\mu,t} H_{\gamma(\eps_t)} (\mu_t|\mu^*_{\eps_t}) := \frac{\d}{\d t} H_{\gamma} (\mu_t | \nu) \rh|_{\nu=\mu^*_{\eps_t},\gamma=\gamma(\eps_t)}
	\end{equation}
	and
	\begin{equation} \label{eq:der2}
		\lf. \dr_{\eps, t} H_{\gamma(\eps_t)} (\mu_t|\mu^*_{\eps_t}) := \frac{\d}{\d t} H_{\gamma(\eps_t)} (\mu | \mu^*_{\eps_t} ) \rh|_{\mu=\mu_t}.
	\end{equation}
	The term \eqref{eq:der1} is from the (instantaneous) evolution of $H_{\gamma(\eps_t)} (\mu_t|\mu^*_{\eps_t})$ along the kinetic diffusion \eqref{eq:kl} for fixed temperature $\eps_t$ and
	the term \eqref{eq:der2} arises from the influence of (instantaneous) invariant measure $\mu^*_{\eps_t}$ and $\gamma(\eps_t)$ on $H_{\gamma(\eps_t)} (\mu_t|\mu^*_{\eps_t})$.

	For the term \eqref{eq:der1}, with a carefully chosen $\eps_t \longmapsto \gamma(\eps_t)$ and by adapting a computation strategy in Ma, Chatterji, Cheng, Flammarion, Bartlett and Jordan \cite{MaBEJ21} for standard kinetic Langevin dynamic,
	we obtain
	\begin{equation} \label{eq:der3}
		\dr_{\mu,t} H_{\gamma(\eps_t)} (\mu_t|\mu^*_{\eps_t})
		~\le~
		- c(\eps_t) \rho_{\eps_t} H_{\gamma(\eps_t)} (\mu_t|\mu^*_{\eps_t}),
	\end{equation}
	where $c(\eps_t) = \gamma^{-1}(\eps_t)$ and $\rho_{\eps_t} > 0$ is the constant in the log-Sobolev inequality (LSI) satisfied by $\mu^*_{\eps_t}$.

	\begin{rmk}
		As in \cite{MonmarchePTRF18} and  \cite{TangZhou21}, it is crucial to use a fine estimation on the constant $\rho_{\eps_t}$ in the LSI for $\mu^*_{\eps_t}$ (under Assumption~\ref{assumption1}),
		in particular for the case with small $\eps_t $.
		We refer to Menz and Schlichting \cite{MenzSchlichtingAOP14}, especially, Section 2.3.3 therein, for this issue.
	\end{rmk}

	The treatment of the term \eqref{eq:der2} is classical (see e.g. Holley, Kusuoka and Stroock \cite{HolleyStroockJFA89}, Miclo \cite{Miclo92}, Tang and Zhou \cite{TangZhou21} in the case of the overdamped simulated annealing, or Monmarché \cite{MonmarchePTRF18} in the case of an alternative kinetic simulated annealing).
	First, we will apply the Lyapunov function technique as in Talay \cite{Talay02} to obtain a uniform bound on the moment of $(X_t, Y_t)$ in \eqref{eq:kl} for all $t \ge 0$.
	Using this uniform moment bound, together with the explicit expression of $\mu^*_{\eps_t}$ for all $t \ge 0$, we can directly compute that, for some sub-exponential $\omega(\eps)$,
	\begin{equation} \label{eq:der4}
		\dr_{\eps, t} H_{\gamma(\eps_t)} (\mu_t|\mu^*_{\eps_t})
		~\le~
		\omega(\eps_t) |\eps_t'| \big(1+H_{\gamma(\eps_t)} (\mu_t|\mu^*_{\eps_t}) \big) , \quad \mbox{for all}\;\; t>0.
	\end{equation}
	where $\eps'_t$ is the derivative of $t \longmapsto \eps_t$.
	The term $\omega(\eps_t)  |\eps_t'|$ at the right-hand-side(r.h.s.) of \eqref{eq:der4} is positive.
	In order to make $H_{\gamma(\eps_t)} (\mu_t|\mu^*_{\eps_t})$ decrease along the time and to get an explicit decay rate,
	we need to make sure that the term $c(\eps_t) \rho_{\eps_t} $ in  \eqref{eq:der3} is greater than the term $\omega(\eps_t)  |\eps_t'|$ in \eqref{eq:der4}.
	This requires a careful choice of $\eps \longmapsto \gamma(\eps)$ as well as some conditions on the cooling schedule $t \longmapsto \eps_t$ as in Assumption~\ref{assumption2}.

	\begin{rmk} \label{rem:compare_monmarche}
	
		$\mathrm{(i)}$ The above sketch of proof, as well as the use of the hypocoercivity method with the distorted entropy \eqref{disent}, is similar to that in Monmarch\'e \cite{MonmarchePTRF18} for a different kinetic annealing  process.
		More precisely, \cite{MonmarchePTRF18} studies the process
		\begin{equation} \label{eq:klMon}
 			X_t = X_0 + \int_0^t Y_s \d s,
			\;\;\;\;
			Y_t = Y_0 - \int_0^t \gr U(X_s)\d s - \int_0^t \frac{1}{\eps_s} Y_s\d s + \int_0^t \sqrt{2} \d B_s,
		\end{equation}
		with $\eps_s \longrightarrow 0$ as $s \longrightarrow \infty$.
		As our cooling schedule is on the diffusion parameter, we need to design another function $\eps_t \longmapsto \gamma(\eps_t)$ to compute the distorted entropy $H_{\gamma(\eps_t)} (\mu_t|\mu^*_{\eps_t})$.
		More importantly, \cite{MonmarchePTRF18} reformulated the problem into a Bakry-Emery framework to establish a contraction inequality similar to \eqref{eq:der3},
		while we apply a different computation strategy adapted from \cite{MaBEJ21}.
		In particular, as shown in \cite{MaBEJ21}, this computation strategy can be adapted to the discrete time setting for convergence analysis of numerical schemes, a subject that is not addressed in \cite{MonmarchePTRF18}.

		\noindent $\mathrm{(ii)}$ From a numerical point of view, our kinetic annealing process \eqref{eq:kl} seems to be more convenient for discrete time simulation than that in \eqref{eq:klMon}.
		Indeed, as $\eps_s \longrightarrow 0$, the Lipschitz constant of the coefficient functions in SDE \eqref{eq:klMon} will explode,
		so the corresponding time discretization error in the numerical analysis can be much harder to control.
	\end{rmk}

	For the convergence analysis of the discrete time process $(\Xb, \Yb)$ in \eqref{eq:diskl}, the main idea will be quite similar.
	Denote $\mub_t := \L (\Xb_t, \Yb_t)$.
	For ease of presentation, let us abbreviate $(\eps_{T_k}, \mub_{T_k}, \mu^*_{\eps_{T_k}})$ to $(\eps_k, \mub_k, \mu^*_k)$.
	Then, with the same distorted entropy function defined by \eqref{eq:der}, we need to compute the difference, for each $k \ge 0$,
	\begin{eqnarray}
		&&
		H_{\gamma(\eps_{k+1})} (\mub_{k+1}|\mu^*_{k+1}) - H_{\gamma(\eps_k)} (\mub_k | \mu^*_k)
		\nonumber \\
		&=&
		H_{\gamma(\eps_k)} (\mub_{k+1}|\mu^*_k) ~-~ H_{\gamma(\eps_k)} (\mub_k|\mu^*_k)
		\label{eq:derdis1} \\
		&&+~
		H_{\gamma(\eps_{k+1})} (\mub_{k+1}|\mu^*_{k+1}) ~-~ H_{\gamma(\eps_k)} (\mub_{k+1}|\mu^*_k).
		 \label{eq:derdis2}
	\end{eqnarray}

	For the term \eqref{eq:derdis1}, we can adapt the computation in \cite{MaBEJ21} for the standard kinetic Langevin process.
	Concretely, we can interpolate $H_{\gamma(\eps_k)} (\mub_k | \mu^*_k)$ and $H_{\gamma(\eps_k)} (\mub_{k+1} | \mu^*_k)  $ by the function $ t \longmapsto H_{\gamma(\eps_k)} (\mub_t| \mu^*_k)$ with $t \in [T_k, T_{k+1}]$.
	Notice that $(\mub_t)_{t \in [T_k, T_{k+1}]}$ satisfies a Fokker-Planck equation with a (partially) frozen coefficient function.
	We can then apply the same computation strategy as in the continuous time setting to obtain a (asymptotic) contraction estimation.

	Similarly, the term  \eqref{eq:derdis2} can be handled by interpolating $H_{\gamma(\eps_k)} (\mub_{k+1}|\mu^*_k)$ and $H_{\gamma(\eps_{k+1})} (\mub_{k+1}|\mu^*_{k+1})$ with $\eps \longrightarrow H_{\gamma(\eps)}(\mub_{k+1} | \mu^*_{\eps})$ for $\eps \in [\eps_{k+1} ,\eps_k]$, together  with a uniform moment estimation on $(\Xb, \Yb)$.
	Finally, with a good choice of the time step size parameters $(\D t_k)_{k \ge 0}$, we can combine the estimations of \eqref{eq:derdis1} and  \eqref{eq:derdis2} to obtain a convergence rate result for the process $(\Xb, \Yb)$ on the discrete time grid $(T_k)_{k \ge 0}$.

	In particular, the above analysis on $H_{\gamma(\eps_k)} (\mub_k | \mu^*_k)$ extends the hypocoercivity method for discrete time Langevin dynamic in \cite{MaBEJ21}
	to the setting with time-dependent coefficient.
	We also notice that most existing papers on the numerical aspects of kinetic type equations by hypocoercivity method mainly focused on the discretization of corresponding Fokker-Planck equations, as in Dujardin, H\'erau and Lafitte \cite{DujardinHerauLafitte20} and Porretta and Zuazua \cite{PorrettaZuazua17}.


\section{Convergence of the continuous-time simulated annealing}
\label{s:conSA}

	In this section, we will present the convergence analysis of continuous-time kinetic annealing process \eqref{eq:kl} in order to prove Theorem \ref{thm1}.
	As discussed in Section~\ref{s:proofsketch}, our main strategy consists in studying the evolution of
	$$
		t \longmapsto H_{\gamma(\eps_t)} (\mu_t | \mu^*_t),
		~~\mbox{with}~
		\gamma(\eps) := \frac{4(1+L^2)}{\eps},
	$$
	where $L > 0$ is the Lipschtiz constant of $\nabla U$ in Assumption \ref{assumption1}.
	For this purpose, we will study separately the two terms at the r.h.s. of \eqref{eq:der}.

	For simplicity of presentation, we will assume that $\theta = 1$. Denote $z = (x,y)$ as a single variable and $Z = (X, Y)$ as the process satisfying \eqref{eq:kl}.
	For each $t \ge 0$, denote by $\mu_t = \L (Z_t)$ the marginal distribution of $Z$.
	Similar to \cite[Proposition 4]{MonmarchePTRF18}, under  the smoothness and growth conditions of $U$ in Assumption \ref{assumption1},  $\mu_t$ has a strictly positive smooth density function on $\R^{2d}$, denoted by $p_t$.
	Recall also that, for each $\eps > 0$, the  invariant probability measure of \eqref{eq:kl} with constant temperature $\eps_t \equiv \eps$ is denoted by $\mu^*_{\eps}$, i.e.,
	with some renormalized constant $C_{\eps} > 0$,
	$$
		\mu^*_{\eps} (\d z) = p^*_{\eps}(z) dz,
		~~\mbox{with}~ p^*_{\eps}(x,y) := C_{\eps} \exp \lf( - \big(U(x)+ |y|^2/2 \big)/ \eps \rh).
	$$

\subsection{Preliminary analysis of continuous kinetic annealing process}
\label{s:conprocess}

	Let us first recall a fine result on the log-Sobolev inequality from \cite[Proposition 2]{MonmarchePTRF18}, which is mainly based on Menz and Schlichting \cite[Corollary 2.18]{MenzSchlichtingAOP14}.

	\begin{prop} \label{prop:LSI}
		Let $U$ satisfy Assumption~\ref{assumption1}.
		Then, for all $\eps > 0$,
		$$
			\KL(\nu|\mu_{\eps}^*) \le \max \left\{ \frac{\eps}{2}, \frac{1}{2\rho_{\eps}} \right\} I(\nu|\mu_{\eps}^*),
			~\mbox{for all}~
			\nu \ll \mu_{\eps}^*,
		$$
		where $\rho_{\eps} =\chi(\eps)\exp \left( - \frac{E_*}{\eps} \right)$ and $\chi(\eps)$ is sub-exponential in the sense that $\eps \log \chi(\eps) \longrightarrow 0 $ as $\eps \longrightarrow 0$.
\end{prop}

	Next, we give a moment estimation on the solution $(X_t, Y_t)_{t \ge 0}$  to \eqref{eq:kl}.
	\begin{prop} \label{prop:conmoment}
		Let Assumptions~\ref{assumption1} and \ref{assumption2} hold. Then,
		$$
			C_0 ~:=~ \sup_{t\ge 0} \E \left[ U(x) + |Y_t|^2   \right] ~<~ \infty .
		$$
	\end{prop}
	\begin{proof}
	Inspired by \cite[Lemma 7]{MonmarchePTRF18}, we consider the Lyapunov function
	$$
		R(x,y) := U(x) + \phi_1(x,y) + \beta \phi_2(x,y),
	$$
	where $\phi_1(x,y) := \frac{|y|^2}{2}$, $\phi_2(x,y) :=  x\cdot y $, and $\beta := \frac{1}{2} \frac{1}{1+\frac{1}{2r}} \frac{1}{1+\frac{1}{\sqrt{r/3}}} < \frac{1}{2}$ with the constant $r > 0$ as given in Assumption~\ref{assumption1}.

	For any smooth function $\phi: \mathbb{R}^d \times \mathbb{R}^d \longrightarrow  \mathbb{R}$,
	let $\CL_{\eps} \phi(x,y) := y\cdot \gr_x \phi(x,y)  - (\gr_x U + y) \cdot \gr_y \phi(x,y) + \eps \Delta_y \phi(x,y) $,
	which corresponds to the generator of the diffusion process \eqref{eq:kl} at temperature $\eps > 0$.
	Then
	$$ \CL_{\eps}  \big( U(x)+ \phi_1(x,y) \big) = -|y|^2 +\eps d. $$
	Further, by Assumption~\ref{assumption1},
	$$
		 \CL_{\eps} \phi_2(x,y)
		 ~=~ |y|^2 - x\cdot y - x \cdot \gr_x U(x)
		 ~\le~
		 \left( 1+ \frac{1}{4c_1} \right) |y|^2 + (c_1-r)|x|^2 +m
	$$
	for any constant $c_1>0$.
	Choosing $c_1=\frac{r}{2}$, and by \eqref{quadratic} as well as the fact that  $\beta<\frac{1}{2}$, it follows that
	\begin{eqnarray*}
		\CL_{\eps} R(x,y) & \le & -|y|^2 + \beta \left( 1+\frac{1}{4c_1}  \right) |y|^2 + \beta (c_1 -r) |x|^2 + \beta m + \eps d\\
		& = & -\left( 1-\beta\left( 1 + \frac{1}{2r}  \right)  \right) |y|^2 - \beta \frac{r}{2} |x|^2 + \beta m +\eps d \quad \\
		& \le & -c_3\beta \lf( U(x) + |y|^2\rh) +c_2,
	\end{eqnarray*}
	for some positive constants $c_2, c_3$ independent of $\eps$.
	Using again $\beta<\frac{1}{2}$ and that $U$ is of quadratic growth (see \eqref{quadratic}),
	we obtain that
	$$
		R(x,y)
		\le
		U(x) + \frac{|y|^2}{2} + \frac{1}{2} x\cdot y
		\le
		c_4\lf( 1+U(x) +|y|^2\rh),
	$$
	for some constant $c_4 > 0$.
	Therefore, for some positive constants $c_5$ and $c_6$ independent of $\eps > 0$,
	we have
	$$
		\CL_{\eps} R(x,y)
		~\le~
		 -c_3\beta \lf( \frac{R(x,y)}{c_4} -1\rh) +c_2
                 ~\le~ -c_5 R(x,y) +c_6.
	$$
	Notice that the Lyapunov function $R(x,y)$ does not depend on $\eps > 0$, so
	$$
		\frac{\d}{\d t}\E[R(X_t,Y_t)] = \E [\CL_{\eps_t} R(X_t, Y_t) ] \le -c_5 \E[R(X_t,Y_t)] + c_6.
	$$
	We can then apply the Gr\"{o}nwall  Lemma to conclude that $(\E[R(X_t,Y_t)])_{t \ge 0}$ is uniformly bounded.

	To conclude, we notice that (see also \cite[Lemma 7]{MonmarchePTRF18}) there is some constant $C > 0$ such that
	$$
		U(x) + |y|^2 \le C \big( R(x,y) + 1 \big),
		~\mbox{for all}~
		(x,y) \in \mathbb{R}^d \times \mathbb{R}^d.
	$$
	Therefore,
	$$
		C_0 ~:=~ \sup_{t\ge 0} \E \left[ U(x) + |Y_t|^2   \right]
		~\le~
		C + C \sup_{t \ge 0} \E \big[ R(X_t, Y_t) \big]
		~<~ \infty .
	$$
\end{proof}


\subsection{Evolution of distorted entropy with fixed temperature}
\label{s:condissipation}

	We now compute $\dr_{\mu, t} H_{\gamma(\eps_t)} (\mu_t | \mu^*_{\eps_t}) $ in \eqref{eq:der1}.
	Although the computation is mainly adapted from Ma, Chatterji, Cheng, Flammarion, Bartlett and Jordan \cite{MaBEJ21}, we nevertheless provide most of details for completeness.
	For simplicity, for a function $f: \R^d \times \R^d \longrightarrow \R$, we abbreviate $\int_{\R^d \times \R^d} f(z) dz$ to $\int f dz$.
	Recalling that $p_t$ (resp. $p^*_{\eps_t}$) represents the density function of $\mu_t$ (resp. $\mu^*_{\eps_t}$), we also write
	$$
		\E_{\mu_t} [f] := \int f p_t dz,
		~\;\;\;\;
		\E_{\mu^*_{\eps_t}} [f] := \int f p^*_{\eps_t} dz.
	$$

	Notice that the marginal distribution $p_t$ satisfies the kinetic Fokker-Planck equation
\begin{equation}
\label{eq:kFP}
    \dr_t p_t = -y\cdot\gr_x p_t + \gr_x U\cdot \gr_y p_t + \gr_y \cdot \left( yp_t  \right) + \eps_t \Delta_y p_t.
\end{equation}
Denote $\gr_z=(\gr_x,\gr_y)^{\T}$. Because $p_{\eps_t}^*\propto\exp\left( -\frac{1}{\eps_t}\left(U(x)+\frac{|y|^2}{2}  \right)   \right) $,
we can rewrite \eqref{eq:kFP} as the following equivalent form:
$$ \dr_t p_t +\gr_z \cdot \left(  p_t
\left( \begin{array}{c}
        -\eps_t\gr_y\log p^*_{\eps_t}  \\
        \eps_t \gr_x\log p_{\eps_t}^* - \eps_t \gr_y\log\frac{p_t}{p^*_{\eps_t}}
        \end{array}\right)  \right) =0 .  $$
To simplify the computation afterwards, we follow the idea in \cite{MaBEJ21} to further add a divergence-free term $\left( \begin{array}{c}
        \eps_t\gr_y\log p_t  \\
        -\eps_t \gr_x\log p_t
        \end{array}\right) $
into the equation. By direct computation, it follows that
\begin{equation}
    \label{FP1}
    \dr_t p_t + \gr_z \cdot \left( p_t v_t  \right) =0,
\end{equation}
where $v_t$ is the vector flow that transports $\mu_t$ towards $\mu^*_{\eps_t}$:
\begin{equation}
    \label{vel1}
    v_t=\eps_t
    \left( \begin{array}{cc}
      0   &  I_d \\
      -I_d   &  -I_d
    \end{array}
    \right)
    \left( \begin{array}{c}
      \gr_x \log \frac{p_t}{p_{\eps_t}^*}  \\
      \gr_y \log \frac{p_t}{p_{\eps_t}^*}
    \end{array}
    \right).
\end{equation}

Let $h_t :=\sqrt{\frac{p_t}{p^*_{\eps_t}}}$ and $S :=\lf(\begin{array}{cc}
    I_d & I_d  \\
    I_d & I_d
\end{array}\rh)$. Then the distorted relative entropy $ H_{\gamma(\eps_t)} (\mu_t | \mu^*_{\eps_t}) $ can be written as
\begin{eqnarray}
    \label{condisent}
    H_{\gamma(\eps_t)} (\mu_t|\mu^*_{\eps_t}) & = & \int \lf(  \lf\langle \gr_z\log\frac{p_t}{p^*_{\eps_t}}, S\,\gr_z\log \frac{p_t}{p^*_{\eps_t}}\rh\rangle + \gamma(t)\log \frac{p_t}{p^*_{\eps_t}} \rh) p_t \d z \no \\
     & = & 4\;\E_{\mu_t}\lf[\lf\langle \gr_z\log h_t, S\,\gr_z\log h_t \rh\rangle    \rh] + 2\gamma(t)\;\E_{\mu_t}\lf[\log h_t   \rh].
\end{eqnarray}

Denote as $\gr_z^*=(\gr_x^*, \gr_y^*) := \lf(-\gr_x\cdot + \frac{1}{\eps_t}\gr_x U\cdot\, , \, -\gr_y\cdot + \frac{y}{\eps_t} \cdot\rh)$ the adjoint operator of $\gr_z$ w.r.t. $\mu^*_{\eps_t}$, in the sense that for any $f,g\in L^2(\mu^*_{\eps_t})$, $\;\E_{\mu^*_{\eps_t}}[f\cdot\gr_z g] = \;\E_{\mu^*_{\eps_t}} [g\gr^*_z f] $.
	By \eqref{FP1}, we obtain that
\begin{eqnarray*}
 \dr_{\mu,t} H_{\gamma(\eps_t)} (\mu_t|\mu^*_{\eps_t}) & := & \lf.\frac{\d}{\d t} H_{\gamma} (\mu_t | \nu) \rh|_{\nu=\mu^*_{\eps_t},\gamma=\gamma(\eps_t)} \\
    & = & \int\lf. \frac{\delta H_{\gamma} (\mu_t | \nu)}{\delta \mu_t}\rh|_{\nu=\mu^*_{\eps_t},\gamma=\gamma(\eps_t)}  \dr_t p_t\, \d z\\
    & = & - \int \frac{\delta H_{\gamma(\eps_t)}(\mu_t | \mu^*_{\eps_t}) }{\delta \mu_t} \gr_z\cdot \lf( p_t v_t \rh)\, \d z \\
    & = & \int \gr_z \lf( \frac{\delta H_{\gamma(\eps_t)}(\mu_t | \mu^*_{\eps_t}) }{\delta \mu_t} \rh) \cdot v_t \, p_t\d z,
\end{eqnarray*}
where $\frac{\delta H_{\gamma(\eps_t)}(\mu_t | \mu^*_{\eps_t}) }{\delta \mu_t}$ is the first order variational derivative of $H_{\gamma(\eps_t)}(\mu_t | \mu^*_{\eps_t})$ at $\mu_t$. By Proposition 2.22 in \cite{LelievreStoltz16}, we have
$$\frac{\delta H_{\gamma(\eps_t)}(\mu_t | \mu^*_{\eps_t})}{\delta \mu_t} = \frac{4}{h_t}\gr_z^* S \gr_z h_t + 2\gamma(\eps_t)\log h_t + \gamma(\eps_t). $$
Then using \eqref{vel1} and \eqref{condisent}, a simple calculation leads to
\begin{eqnarray}
       & &  \dr_{\mu,t} H_{\gamma(\eps_t)} (\mu_t|\mu^*_{\eps_t})\no\\
       & = & \int \gr_z \lf( \frac{4}{h_t}\gr_z^* S \gr_z h_t + 2\gamma(\eps_t)\log h_t  \rh) \cdot v_t \, p_t\d z \no \\
       & = & \label{ex1} -4\eps_t\gamma(\eps_t)\int \lf\langle \gr_z h_t\, , \,  \begin{pmatrix}
      0   &  -I_d \\
      I_d   &  I_d
       \end{pmatrix} \gr_z h_t \rh\rangle  p^*_{\eps_t}\d z \\
    &  & \label{ex2} + 8\eps_t\int \lf\langle \gr_z h_t\, , \, \begin{pmatrix}
      0   &  -I_d \\
      I_d   &  I_d
       \end{pmatrix}  \gr_z h_t \rh\rangle \frac{\gr_z^*S\gr_z h_t}{h_t} \,p^*_{\eps_t}\d z\\
       & & \label{ex3} -8\eps_t\int \lf\langle \gr_z\gr_z^*S\gr_z h_t\, , \, \begin{pmatrix}
      0   &  -I_d \\
      I_d   &  I_d
       \end{pmatrix}  \gr_z h_t\rh\rangle \, p^*_{\eps_t}\d z.
\end{eqnarray}

For \eqref{ex1}, it equals to
\begin{eqnarray}  \label{ex11}
     \;\;\;\;\;
    -4\eps_t \gamma(\eps_t)\, \;\E_{\mu^*_{\eps_t}}  \big[ \big|\gr_y h_t \big|^2 \big]
       ~=~
       -\eps_t \gamma(\eps_t)\, \;\E_{\mu_t} \Big[ \Big|\gr_y\log\frac{p_t}{p^*_{\eps_t}} \Big|^2 \Big].
\end{eqnarray}
In addition, \eqref{ex2} can be simplified as
\begin{eqnarray}
       &   & 8\eps_t\int \lf\langle \gr_z h_t\, , \, \begin{pmatrix}
      0   &  -I_d \\
      I_d   &  I_d
       \end{pmatrix}  \gr_z h_t \rh\rangle \frac{\gr_z^*S\gr_z h_t}{h_t} \,p^*_{\eps_t}\d z \no \\
       & = & 8\eps_t\, \;\E_{\mu^*_{\eps_t}} \lf[ \frac{|\gr_y h_t|^2}{h_t} \gr_z^*S\gr_z h_t \rh] \no\\
      & = & 8\eps_t \;\E_{\mu^*_{\eps_t}} \lf[ \lf\langle \frac{\gr_z|\gr_y h_t|^2}{h_t} \, , \, S\gr_z h_t \rh\rangle \rh]
      - 8\eps_t \;\E_{\mu^*_{\eps_t}} \lf[ \lf\langle \frac{|\gr_y h_t|^2}{h_t^2}\gr_z h_t \, , \, S\gr_z h_t \rh\rangle \rh]     \no   \\
      & = & \label{ex21} 16\eps_t \;\E_{\mu^*_{\eps_t}} \lf[ \lf\langle \frac{\gr_z h_t}{h_t} (\gr_y h_t)^{\T} \, , \, S\gr_z\gr_y h_t \rh\rangle_F \rh] \\
      &  & - 8\eps_t \;\E_{\mu^*_{\eps_t}} \lf[ \lf\langle \frac{\gr_z h_t}{h_t} (\gr_y h_t)^{\T} \, , \, S \frac{\gr_z h_t}{h_t} (\gr_y h_t)^{\T} \rh\rangle_F \rh]. \no
\end{eqnarray}
For \eqref{ex3}, we write it as the sum of the following three terms:
\begin{eqnarray*}
     I_1  & := & \label{ex31} -8\eps_t\; \;\E_{\mu_{\eps_t}^*} \lf[ \lf\langle \gr_z (\gr_x^*\gr_y + \gr_y^*\gr_x) h_t\, , \, \begin{pmatrix}
      0   &  -I_d \\
      I_d   &  I_d
       \end{pmatrix}  \gr_z h_t\rh\rangle \rh], \\
    I_2  & := & \label{ex32} -8\eps_t \; \;\E_{\,u_{\eps_t}^*} \lf[ \lf\langle \gr_z \gr_x^*\gr_x  h_t\, , \, \begin{pmatrix}
      0   &  -I_d \\
      I_d   &  I_d
       \end{pmatrix}  \gr_z h_t\rh\rangle \rh],\\
      I_3 & := & \label{ex33} -8\eps_t \; \;\E_{\mu_{\eps_t}^*} \lf[ \lf\langle \gr_z \gr_y^*\gr_y  h_t\, , \, \begin{pmatrix}
      0   &  -I_d \\
      I_d   &  I_d
       \end{pmatrix}  \gr_z h_t\rh\rangle \rh].
\end{eqnarray*}
	Further, by \cite[Lemma 9]{MaBEJ21}, we have
	\begin{eqnarray*}
           I_1 & = & 8\; \;\E_{\mu^*_{\eps_t}} \lf[\langle \gr_y h_t, \gr_x^2 U\gr_y h_t\rangle - \langle \gr_x h_t, \gr_x h_t + \gr_y h_t  \rangle   \rh]\\
            &  & -16\eps_t \; \;\E_{\mu^*_{\eps_t}} \lf[ \langle \gr_y\gr_x h_t , \gr_y^2 h_t \rangle_F   \rh], \\
       I_2 &=& 8\; \;\E_{\mu_{\eps_t}^*} \lf[ \langle \gr_x h_t, \gr_x^2 U \gr_y h_t \rangle \rh] - 8\eps_t \; \;\E_{\mu_{\eps_t}^*} \lf[\langle \gr_x \gr_y h_t, \gr_x \gr_y h_t \rangle_F  \rh], \\
        I_3  &=&  -8 \; \;\E_{\mu^*_{\eps_t}} \lf[ \langle \gr_x h_t + \gr_y h_t , \gr_y h_t \rangle + \eps_t \langle \gr_y^2 h_t, \gr_y^2 h_t \rangle_F \rh].
    \end{eqnarray*}
	Consequently, we obtain that \eqref{ex3} equals 
	\begin{multline*}
		- 8\;\;\E_{\mu_{\eps_t}^*} \lf[ \lf\langle \gr_x h_t , \gr_x h_t + (2I_d-\gr_x^2 U) \gr_y h_t  \rh\rangle \rh] \\
		 - 8  \; \;\E_{\mu^*_{\eps_t}} \lf[ \eps_t \lf\langle \gr_z\gr_y h_t, S\gr_z\gr_y h_t  \rh\rangle_F  + \lf\langle \gr_y h_t, (I_d-\gr_x^2 U) \gr_y h_t    \rh\rangle  \rh].
	\end{multline*}

	Combining the above with \eqref{ex11} and \eqref{ex21}, we derive
\begin{eqnarray*}
       &  & \dr_{\mu,t} H_{\gamma(\eps_t)} (\mu_t|\mu^*_{\eps_t})\\
       & = & -8\eps_t \;\E_{\mu_{\eps_t}^*} \Big[ \big\langle
       \gr_z\gr_y h_t - (\gr_z\log h_t) (\gr_y h_t)^{\T} ,
       S \lf( \gr_z\gr_y h_t - (\gr_z\log h_t) (\gr_y h_t)^{\T}  \rh)
       \big \rangle_F    \Big]\\
       & & -4 \;\E_{\mu_{\eps_t}^*} \lf[\lf\langle \gr_z h_t, M_t \gr_z h_t \rh\rangle  \rh] \\
       & = & -8 \eps_t \;\E_{\mu_t} \lf[ \lf\langle
       \gr_z\gr_y \log h_t , S \gr_z\gr_y \log h_t
       \rh\rangle_F    \rh] -4 \;\E_{\mu_t} \lf[\lf\langle \gr_z \log h_t, M_t \gr_z \log h_t \rh\rangle  \rh],
\end{eqnarray*}
with $M_t=\begin{pmatrix}
 2I_d & 2I_d -\gr_x^2 U\\
 2I_d - \gr_x^2 U & \lf( 2+ \gamma(\eps_t)\eps_t  \rh)I_d - 2\gr_x^2 U
\end{pmatrix} \in\M_{2d\times 2d}(\R) $.
Because $S=\begin{pmatrix}
 I_d & I_d\\
 I_d & I_d
\end{pmatrix}
\succeq 0 $, we obtain
can get rid of the term involving with the second-order derivatives and obtain that
\begin{eqnarray*}
       \dr_{\mu,t} H_{\gamma(\eps_t)} (\mu_t|\mu^*_{\eps_t}) & \le & -4 \;\E_{\mu_t} \lf[\lf\langle \gr_z \log h_t, M_t \gr_z \log h_t \rh\rangle  \rh]  \\
       & = & - \;\E_{\mu_t} \lf[\lf\langle \gr_z \log \frac{p_t}{p_{\eps_t}^*}, M_t \gr_z \log \frac{p_t}{p_{\eps_t}^*} \rh\rangle  \rh].
\end{eqnarray*}

	Note that the r.h.s. of the above inequality resembles the Fisher information $I(\mu_t|\mu^*_{\eps_t}) = \;\E_{\mu_t} \lf[\lf\langle \gr_z \log \frac{p_t}{p_{\eps_t}^*} , \gr_z \log \frac{p_t}{p_{\eps_t}^*} \rh\rangle  \rh]$.
	In order to use the log-Sobolev inequality satisfied by $\mu^*_{\eps_t}$ (Proposition~\ref{prop:LSI}) and connect it to the distorted relative entropy \eqref{condisent},
	we will prove in Lemma \ref{Mpsd} that, for sufficiently large $t \ge 0$,
	\begin{equation} \label{eq:claimMt}
		M_t ~\succeq~ c(\eps_t)\rho_{\eps_t} \lf( S  + \frac{\gamma(\eps_t)}{2\rho_{\eps_t}} I_{2d} \rh),
	\end{equation}
	where $c(\eps_t) := \frac{1}{\gamma(\eps_t)} = \frac{\eps_t}{4\lf( 1+ L^2 \rh)}$ and $\rho_{\eps_t}$ is given in Proposition~\ref{prop:LSI}.
	It follows that
	\begin{eqnarray*}
		&&\!\!\!
		\dr_{\mu,t} H_{\gamma(\eps_t)} (\mu_t|\mu^*_{\eps_t}) \\
		\!\!\! & \le &\!\!\! - c(\eps_t)\rho_{\eps_t} \E_{\mu_t} \!\! \lf[\lf\langle \gr_z \log \frac{p_t}{p_{\eps_t}^*}, S \gr_z \log \frac{p_t}{p_{\eps_t}^*} \rh\rangle  + \frac{\gamma(\eps_t)}{2\rho_{\eps_t}} \lf\langle \gr_z \log \frac{p_t}{p_{\eps_t}^*} , \gr_z \log \frac{p_t}{p_{\eps_t}^*} \rh\rangle \! \rh]\\
		& \le &\!\!\! - c(\eps_t)\rho_{\eps_t} \;\E_{\mu_t} \lf[\lf\langle \gr_z \log \frac{p_t}{p_{\eps_t}^*}, S \gr_z \log \frac{p_t}{p_{\eps_t}^*} \rh\rangle  + \gamma(\eps_t) \KL (\mu_t | \mu^*_{\eps_t}) \rh]\\
		& = &\!\!\! - c(\eps_t)\rho_{\eps_t} H_{\gamma(\eps_t)}(\mu_t|\mu^*_{\eps_t}).
	\end{eqnarray*}

	\begin{lem} \label{Mpsd}
		Let Assumptions~\ref{assumption1} and \ref{assumption2} hold true.
		Then, for sufficiently large $t$, $M_t$ satisfies \eqref{eq:claimMt}.
	\end{lem}
	\begin{proof}
	The proof is very similar to that of \cite[Lemma 8]{MonmarchePTRF18}, so we only present the main idea here.

	By the choice of $\gamma(\eps_t)$ and $c(\eps_t)$, we have that
	$(\eps_t, c(\eps_t) ,\rho_{\eps_t}, \gamma(\eps_t)) \longrightarrow (0, 0, 0, \infty)$ as $t \longrightarrow \infty$.
	Thus, it is sufficient to prove that for sufficiently large $t \ge 0$,
	$$
		M_t -\frac{1}{2} c(\eps_t)\gamma(\eps_t) I_{2d}
		~=~
		M_t -\frac{1}{2} I_{2d}
		~\succ~
		0 .
	$$
	We first compute the characteristic equation of $M_t-\frac{1}{2} I_{2d}$:
	\begin{eqnarray*}
		&&
		\det \lf( M_t -\frac{1}{2} I_{2d} -\lambda I_{2d} \rh) \\
		&=&
		\det \begin{pmatrix}
		\lf(\frac{3}{2} -\lambda \rh) I_{d} & 2I_d -\gr_x^2 U\\
		2I_d -\gr_x^2 U  &  \lf( \frac{11}{2} - \lambda + 4L^2  \rh)I_d -2\gr_x^2 U
		\end{pmatrix}\\
		&=&
		\det \lf( \lf(\frac{3}{2} -\lambda\rh) \lf(\frac{11}{2} -\lambda + 4 L^2 \rh) I_d - 2 \lf(\frac{3}{2}-\lambda\rh) \gr_x^2 U - \lf(2I_d -\gr_x^2 U\rh)^2  \rh).
	\end{eqnarray*}
	By diagonalizing $\gr_x^2 U$, with $(\gr_x^2 U)_i$ denoting the $i$-th eigenvalue of $\gr_x^2 U$,
	we obtain $d$ quadratic equations on $\lambda$ in the form
	\begin{equation*}
		\lambda^2 - \lf( 7+4L^2 -2 (\gr_x^2 U)_i  \rh)\lambda + \lf( \frac{17}{4} + 6L^2 - (\gr_x^2 U)^2_i + (\gr_x^2 U)_i  \rh) =0,
	\end{equation*}
	whose roots are the eigenvalues of $M_t-\frac{1}{2}I_{2d}$.
	Using the fact that $|(\gr_x^2 U)_i| \le L $ (as $\gr U$ is $L$-Lipschitz), we can check that all roots of the above equations are positive,
	so all eigenvalues of $M_t -\frac{1}{2}I_{2d}$ are positive.
	This concludes the proof.
\end{proof}

\begin{rmk}
\label{rmkfiniteH}
	The previous lemma only ensures $M_t \succeq c(\eps_t)\rho_{\eps_t} \lf( S  + \frac{\gamma(\eps_t)}{2\rho_{\eps_t}} I_{2d} \rh)$ for sufficiently large $t$. However, this is sufficient to obtain the convergence of our kinetic annealing process \eqref{eq:kl}, as long as we can ensure the distorted relative entropy $H_{\gamma(\eps_t)}(\mu_t|\mu^*_{\eps_t})$ is finite in any finite time interval (see Lemma~\ref{finiteH} below).
	We also want to point out that the choice of $\gamma(\eps_t)$ and $c(\eps_t) $ is not unique to ensure that $M_t$ satisfies \eqref{eq:claimMt}.
\end{rmk}

	To conclude this subsection, we summarize the above computation in the following proposition:
	\begin{prop} \label{condissipation}
		Let Assumptions~\ref{assumption1} and \ref{assumption2} hold.
		Then, for sufficiently large $t \ge 0$, we have
		\begin{equation*}
			\dr_{\mu, t} H_{\gamma(\eps_t)}(\mu_t|\mu^*_{\eps_t}) \le - c(\eps_t)\rho_{\eps_t} H_{\gamma(\eps_t)}(\mu_t|\mu^*_{\eps_t}),
		\end{equation*}
		where $c(\eps_t) := \frac{\eps_t}{4\lf( 1+ L^2 \rh)}$ and $\rho_{\eps_t}$ is given in Proposition~\ref{prop:LSI}.
	\end{prop}


\subsection{Evolution of distorted entropy with cooling temperature}
\label{s:33}

	In this subsection, we will compute $\dr_{\eps,t} H_{\gamma(\eps_t)}(\mu_t | \mu^*_t) $ as defined in \eqref{eq:der2}.
	Recall that $C_0$ is the upper bound of $\E[ U(X_t) + |Y_t|^2]$ as defined in Proposition \ref{prop:conmoment}.
	Denote by $\eps'_t$ the derivative of $t \longmapsto \eps_t$.

	\begin{lem} \label{conepsilon}
		Let Assumptions~\ref{assumption1} and \ref{assumption2} hold true.
		Then, there exists some sub-exponential function $\eps \longmapsto \omega(\eps)$, such that
		\begin{equation*}
			\dr_{\eps,t} H_{\gamma(\eps_t)}(\mu_t|\mu^*_{\eps_t}) \le |\eps_t'| \omega(\eps_t) \lf( H_{\gamma(\eps_t)}(\mu_t|\mu^*_{\eps_t}) + 1 + C_0 \rh).
		\end{equation*}
	\end{lem}

\begin{proof}
	We follow the proof of \cite[Lemma 15]{MonmarchePTRF18}.
	In the following, the constant $C > 0$ may change from line to line.

	Notice that $\mu^*_{\eps_t}$ has smooth density function $p^*_{\eps_t} (x,y) = \frac{1}{Z_{\eps_t}} \exp \lf( -\frac{1}{\eps_t}  \lf( U(x) + \frac{|y|^2}{2}  \rh) \rh) $ with renormalization constant $Z_{\eps_t} > 0$.
	Use \eqref{quadratic}, it is easy to obtain that $\mu^*_{\eps_t}$ has bounded moments uniformly in $\eps$.
	Thus, there exists some constant $C > 0$ such that
\begin{eqnarray*}
       \lf|\dr_{\eps} \log p^*_{\eps_t} (x,y) \rh|
       & = & \lf|\dr_{\eps} \lf( -\log Z_{\eps_t} - \frac{U(x)}{\eps_t} - \frac{|y|^2}{2\eps_{t}}   \rh) \rh|\\
       & = & \lf|\frac{1}{\eps_t^2} \lf( U(x) + \frac{|y|^2}{2}  \rh ) - \int \frac{1}{\eps_t^2} \lf( U(a) + \frac{|b|^2}{2}  \rh ) \d p^*_{\eps_t} (a,b) \rh|\\
       & \le & C \eps_t^{-2} \lf( 1+ U(x) +|y|^2 \rh).
\end{eqnarray*}
Moreover,
\begin{equation*}
    \gr_z \dr_{\eps} \log p^*_{\eps_t} (x,y) = \eps_t^{-2} \begin{pmatrix}
    \gr_x U(x)\\ y
    \end{pmatrix}.
\end{equation*}
Thus, setting $\tilde{S} = (I_d, I_d)$ and using $L$-Lipschitz of $\gr_x U$, we have
\begin{eqnarray*}
       &  & \dr_{\eps} \;\E_{\mu_t} \lf[ \lf\langle \gr_z\log\frac{p_t}{p^*_{\eps_t}}, S\,\gr_z\log \frac{p_t}{p^*_{\eps_t}}\rh\rangle \rh]\\
       & = & \dr_{\eps} \int \lf|\tilde{S} \gr_z \log\frac{p_t}{p^*_{\eps_t}} \rh|^2  p_t \d z \\
       & = & 2 \int \lf( \tilde{S} \gr_z \log\frac{p_t}{p^*_{\eps_t}} \cdot \tilde{S} \gr_z \dr_{\eps} \log\frac{p_t}{p^*_{\eps_t}}\rh) p_t \d z \\
       & = & -2 \int \lf( \tilde{S} \gr_z \log\frac{p_t}{p^*_{\eps_t}} \cdot \tilde{S} \gr_z \dr_{\eps} \log p^*_{\eps_t} \rh) p_t \d z \\
       & \le & \int \lf|\tilde{S} \gr_z \log\frac{p_t}{p^*_{\eps_t}} \rh|^2 p_t \d z + \int \lf| \tilde{S} \gr_z \dr_{\eps} \log p^*_{\eps_t} \rh|^2 p_t \d z \\
       & \le &  \;\E_{\mu_t} \lf[ \lf\langle \gr_z\log\frac{p_t}{p^*_{\eps_t}}, S\,\gr_z\log \frac{p_t}{p^*_{\eps_t}}\rh\rangle \rh] + 2\eps_t^{-4} \E \lf[ |\gr_x U(X_t)|^2 + |Y_t|^2  \rh]\\
       & \le & \;\E_{\mu_t} \lf[ \lf\langle \gr_z\log\frac{p_t}{p^*_{\eps_t}}, S\,\gr_z\log \frac{p_t}{p^*_{\eps_t}}\rh\rangle \rh] + C\eps_t^{-4} (1 + C_0)  
\end{eqnarray*}
and
\begin{eqnarray*}
       \dr_{\eps} \lf( \gamma(\eps_t) \;\E_{\mu_t} \lf[ \log\frac{p_t}{p^*_{\eps_t}}  \rh] \rh)
       & = & \gamma'(\eps_t) \;\E_{\mu_t} \lf[ \log\frac{p_t}{p^*_{\eps_t}} \rh] - \gamma(\eps_t) \;\E_{\mu_t} \lf[ \dr_{\eps} \log p^*_{\eps_t} \rh]\\
       & \le & C \eps_t^{-2} \;\E_{\mu_t} \lf[ \log\frac{p_t}{p^*_{\eps_t}}  \rh] + C \eps_t^{-3} (1 + C_0) .
\end{eqnarray*}
Then, by \eqref{eq:der2}, we have
\begin{eqnarray*}
       \dr_{\eps,t} H_{\gamma(\eps_t)}(\mu_t|\mu^*_{\eps_t})
       & = & \eps_t'\,\dr_{\eps} \;\E_{\mu_t} \lf[ \lf\langle \gr_z\log\frac{p_t}{p^*_{\eps_t}}, S\,\gr_z\log \frac{p_t}{p^*_{\eps_t}}\rh\rangle + \gamma(\eps_t) \log\frac{p_t}{p^*_{\eps_t}} \rh]\\
       & \le & |\eps_t'|\lf( (1+\eps_t^{-1}) H_{\gamma(\eps_t)}(\mu_t|\mu^*_{\eps_t}) + C (\eps_t^{-3} + \eps_t^{-4}) (1 + C_0)  \rh)\\
       & \le & |\eps_t'| \omega(\eps_t) \lf( H_{\gamma(\eps_t)}(\mu_t|\mu^*_{\eps_t}) + 1 + C_0 \rh) 
\end{eqnarray*}
for $\omega(\eps) := 1+\eps^{-1} + C(\eps^{-3} + \eps^{-4}) $ which is clearly sub-exponential.
\end{proof}


\subsection{Proof of Theorem~\ref{thm1}}

	For any $t\ge 0$,  let $\tilde Z_t = (\tilde{X}_t, \tilde{Y}_t)$ be a random variable (on the same initial probability space) following distribution $\mu^*_{\eps_t}$.
	Then, for any $\delta>0$,
	\begin{eqnarray} \label{bound}
		 \P (U(X_t)>\delta) & = & \P(U(X_t)>\delta, ~U(\tilde{X}_t)>\delta) + \P (U(X_t)>\delta, ~U(\tilde{X}_t)\le \delta)  \nonumber \\
		 & \le & \P(U(\tilde{X}_t)>\delta) + \TV (\mu_t | \mu^*_{\eps_t}) \nonumber \\
		 & \le & \P(U(\tilde{X}_t)>\delta) + \sqrt{2\KL(\mu_t | \mu^*_{\eps_t})}  \nonumber\\
		 & \le & \P(U(\tilde{X}_t)>\delta) + \sqrt{ 2 H_{\gamma(\eps_t)}(\mu_t | \mu^*_{\eps_t})},
	\end{eqnarray}
	where we have used the Csiszár-Kullback-Pinsker inequality, the definition of the distorted relative entropy \eqref{disent},
	and the fact that $\gamma(\eps_t) = \frac{4}{\eps_t} (1+L^2) \ge 1$ for sufficiently large $t \ge 0$.
	We then conclude the proof by Lemma~\ref{Laplace} and Proposition~\ref{conannealing}.
	\qed

	For the estimation of $\P(U(\tilde{X_t})>\delta)$, we have the following classical result (see e.g. \cite[Lemma 3]{MonmarchePTRF18}, or \cite[Lemma 3]{TangZhou21}).

	\begin{lem} \label{Laplace}
		Let $U$ satisfy Assumption~\ref{assumption1}.
		Then, for all constants $ \delta$ and $\alpha>0$, there exists some constant $C>0$ (which depends only on $\delta,\, \alpha,\, U,\, d$) such that
		$$\P(U(\tilde{X_t})>\delta) \le C e^{-\frac{\delta-\alpha}{\eps_t}} .  $$
		In particular, let $\eps_t$ satisfy Assumption~\ref{assumption2}. Then for sufficiently large $t > 0$, we have
		$$\P(U(\tilde{X_t})>\delta) \le C t^{-\frac{\delta-\alpha}{E}} .$$
	\end{lem}

	To provide an estimation on $H_{\gamma(\eps_t)}(\mu_t|\mu^*_{\eps_t})$, we first prove that it is uniformly bounded on any finite time horizon $[0,t]$.
	\begin{lem} \label{finiteH}
		Let Assumptions \ref{assumption1} and \ref{assumption2} hold true.
		Then, for all $t \ge 0$, we have $$\sup_{0 \le s \le t} H_{\gamma(\eps_s)}(\mu_s | \mu^*_{\eps_s}) < \infty.$$
	\end{lem}

\begin{proof}
	Recall $H_{\gamma(\eps_t)}(\mu_t|\mu^*_{\eps_t}) = \;\E_{\mu_t}\lf[ \lf| \gr_x\log\frac{\d \mu_t}{\d \mu^*_{\eps_t}} + \gr_y\log\frac{\d \mu_t}{\d \mu^*_{\eps_t}} \rh|^2 + \gamma(\eps_t)\log\frac{\d \mu_t}{\d\mu^*_{\eps_t} } \rh] $, and notice that
	$$
		\E_{\mu_t}\lf[\lf| \gr_x\log\frac{\d \mu_t}{\d \mu^*_{\eps_t}} + \gr_y\log\frac{\d \mu_t}{\d \mu^*_{\eps_t}} \rh|^2 \rh]
		~\le~
		2I(\mu_t|\mu^*_{\eps_t}).
	$$
	Then, by the log-Sobolev inequality, it suffices to prove  that $I(\mu_t|\mu^*_{\eps_t})$ is uniformly bounded on any finite horizon.

	As in \eqref{eq:der}, we write
\begin{equation}
\label{derivativeI}
    \frac{\d}{\d t} I(\mu_t|\mu^*_{\eps_t}) = \dr_{\mu, t} I(\mu_t|\mu^*_{\eps_t}) + \dr_{\eps,t} I(\mu_t|\mu^*_{\eps_t}),
\end{equation}
where
	$$
		\dr_{\mu, t} I(\mu_t|\mu^*_{\eps_t}) :=\lf. \frac{\d}{\d t} I(\mu_t|\nu) \rh|_{\nu=\mu^*_{\eps_t}}
		~\mbox{and}~\;\;\;
		\dr_{\eps, t} I(\mu_t|\mu^*_{\eps_t}) := \lf. \frac{\d}{\d t} I(\nu|\mu^*_{\eps_t}) \rh|_{\nu=\mu_t}.
	$$
Similar to the computation in Section~\ref{s:condissipation}, for $h_t=\sqrt{p_t/p^*_{\eps_t}}$, we have $$I(\mu_t|\mu^*_{\eps_t}) = 4 \;\E_{\mu_t}[\langle \gr_z \log h_t, \gr_z \log h_t\rangle ]. $$
Then,
\begin{eqnarray*}
        \dr_{\mu, t} I(\mu_t|\mu^*_{\eps_t}) & = & \int \gr_z\lf( \frac{4}{h_t}\gr_z^*\gr_z h_t \rh)\cdot v_t p_t \d z\\
        & = & -8 \eps_t \;\E_{\mu^*_{\eps_t}} \lf[\lf|\gr_z\gr_yh_t - (\gr_z\log h_t) (\gr_y h_t)^{\T} \rh|^2  \rh]\\
        &  & -4\;\E_{\mu^*_{\eps_t}}\lf[ \langle \gr_z h_t, \tilde{M} \gr_z h_t\rangle \rh]\\
        & \le & -4\;\E_{\mu^*_{\eps_t}}\lf[ \langle \gr_z h_t, \tilde{M} \gr_z h_t\rangle \rh],
\end{eqnarray*}
 where $v_t$ satisfies \eqref{vel1} and $\tilde{M}=\begin{pmatrix}
 0 & I_d -\gr_x^2 U\\
 I_d - \gr_x^2 U & 2I_d
\end{pmatrix}$.
Using the fact that $\gr U$ is $L$-Lipschitz, it is easy to show that $\tilde{M} \succeq -(L+1) I_{2d} $ and thus
\begin{equation}
    \label{derivativeI3}
     \dr_{\mu, t} I(\mu_t|\mu^*_{\eps_t}) \le (L+1) I(\mu_t | \mu^*_{\eps_t} ).
\end{equation}
	
	Further, a similar calculation to the one in the proof of Lemma~\ref{conepsilon} yields that, for some constant $C > 0$,
\begin{equation*}
    \dr_{\eps} I (\mu_t |\mu^*_{\eps} ) \le I (\mu_t |\mu^*_{\eps} ) + C\eps^{-4} (1 + C_0).
\end{equation*}
This inequality, together with Proposition~\ref{prop:conmoment}, yields the existence of some sub-exponential $\tilde{\omega}( \cdot)$, such that
\begin{equation}
    \label{derivativeI4}
    \dr_{\eps, t} I(\mu_t|\mu^*_{\eps_t}) \le |\eps_t'|\tilde{\omega}(\eps_t) (1+I(\mu_t|\mu^*_{\eps_t})).
\end{equation}
	Plugging \eqref{derivativeI3} and \eqref{derivativeI4} into \eqref{derivativeI}, we obtain
\begin{equation*}
    \frac{\d}{\d t} I(\mu_t | \mu^*_{\eps_t}) \le (1+L) I(\mu_t | \mu^*_{\eps_t}) + |\eps_t'|\tilde{\omega}(\eps_t) (1+I(\mu_t|\mu^*_{\eps_t})).
\end{equation*}
	Because $\tilde{\omega}(\eps_t)$ and $|\eps_t'|$ are all bounded on any finite time interval, there exists $C>0$, such that
\begin{equation*}
    \frac{\d}{\d t} I(\mu_t | \mu^*_{\eps_t}) \le C (I(\mu_t | \mu^*_{\eps_t}) + 1).
\end{equation*}
	Finally, by Assumption~\ref{assumption2}, we conclude $I(\mu_0 | \mu^*_{\eps_0})$ is finite. The proof then completes.
\end{proof}

\begin{prop}
\label{conannealing}
Let Assumptions~\ref{assumption1} and \ref{assumption2} hold and choose $\gamma(\eps_t)=\frac{4}{\eps_t}\lf( 1+ L^2 \rh) $. Then, for any $\alpha>0$, there exists $C>0$, such that for sufficiently large $t$,
\begin{equation*}
    H_{\gamma(\eps_t)}(\mu_t|\mu^*_{\eps_t}) \le C t^{-\lf( 1-\frac{E_*}{E} -\alpha \rh)}.
\end{equation*}

\end{prop}

\begin{proof}

	First, by Lemma \ref{finiteH}, $H_{\gamma(\eps_t)}(\mu_t|\mu^*_{\eps_t})$ is finite for any $t \ge 0$.
	Next, by \eqref{eq:der}, Proposition~\ref{condissipation} and Lemma~\ref{conepsilon}, for sufficiently large $t \ge 0$, we have
	\begin{eqnarray*}
		\frac{\d}{\d t} H_{\gamma(\eps_t)}(\mu_t|\mu^*_{\eps_t})
		~\le~
		- c(\eps_t)\rho_{\eps_t} H_{\gamma(\eps_t)}(\mu_t|\mu^*_{\eps_t}) + |\eps'_t| \omega(\eps_t) \lf( H_{\gamma(\eps_t)}(\mu_t|\mu^*_{\eps_t}) + 1+C_0 \rh).
	\end{eqnarray*}

	Further, by Proposition~\ref{prop:LSI}, we have $\rho_{\eps} =\chi(\eps)\exp \left( - \frac{E_*}{\eps} \right) $ with some sub-exponential function $\chi(\eps)$.
Then, by Assumption~\ref{assumption2} and Proposition~\ref{condissipation}, for sufficiently large $t \ge 0$, we have that for all $\alpha>0$,
\begin{equation*}
    c(\eps_t)\rho_{\eps_t}  = \O \lf(  t^{-\lf( \alpha + \frac{E_*}{E} \rh)} \rh).
\end{equation*}
According to Proposition~\ref{prop:conmoment}, $ \E \left[ U(x) + |Y_t|^2   \right] $ is uniformly bounded, so we obtain
\begin{equation*}
      |\eps'_t| \omega(\eps_t) \lf( H_{\gamma(\eps_t)}(\mu_t|\mu^*_{\eps_t}) + \E \lf[ 1 + U(x) + |Y_t|^2   \rh]  \rh) = \O  \lf(t^{-1} H_{\gamma(\eps_t)}(\mu_t|\mu^*_{\eps_t}) + t^{-1}  \rh).
\end{equation*}
Thus, there exists some positive constants $\tilde{c}_1$ and $\tilde{c}_2$ such that for sufficiently large $t$,
\begin{equation*}
    \frac{\d}{\d t} H_{\gamma(\eps_t)}(\mu_t|\mu^*_{\eps_t})
        \le  -\tilde{c}_1 t^{-\lf( \alpha + \frac{E_*}{E} \rh)} H_{\gamma(\eps_t)}(\mu_t|\mu^*_{\eps_t}) + \tilde{c}_2 t^{-1} H_{\gamma(\eps_t)}(\mu_t|\mu^*_{\eps_t}) + \tilde{c}_2 t^{-1} .
\end{equation*}
Because $E>E_*$, for the terms of $H_{\gamma(\eps_t)}(\mu_t | \mu_{\eps_t}^*)$ in the r.h.s. of the above inequality, 
the first term dominates the second term when $\alpha > 0$ is small. As a result, for some positive constants $c_1$ and $c_2$, we have
\begin{eqnarray*}
    \frac{\d}{\d t} H_{\gamma(\eps_t)}(\mu_t|\mu^*_{\eps_t})
       & \le & -c_1 t^{-\lf( \alpha + \frac{E_*}{E} \rh)} H_{\gamma(\eps_t)}(\mu_t|\mu^*_{\eps_t}) + c_2 t^{-1} \\
       & \le & -c_1 t^{-\lf( \alpha + \frac{E_*}{E} \rh)} H_{\gamma(\eps_t)}(\mu_t|\mu^*_{\eps_t}) + c_2 t^{-1 + \alpha} .
\end{eqnarray*}
Set $\alpha < \frac{1}{2} \lf( 1-\frac{E_*}{E}  \rh)  $, then for sufficiently large $t_0$ and for all $t\ge t_0$, a Gr\"{o}nwall type argument yields that
\begin{equation*}
    H_{\gamma(\eps_t)}(\mu_t|\mu^*_{\eps_t}) \le \frac{2c_2}{c_1} t^{-\lf( 1-\frac{E_*}{E} -2\alpha  \rh)} + H_{\gamma(\eps_{t_0})}(\mu_{t_0}|\mu^*_{\eps_{t_0}}) \exp \lf( -\frac{c_1}{\kappa} (t^{\kappa} - t_0^{\kappa} ) \rh) ,
\end{equation*}
where $\kappa = 1-\frac{E_*}{E} -\alpha >0 $.

	Finally, by Lemma~\ref{finiteH},
	we have $H_{\gamma(\eps_{t_0})}(\mu_{t_0}|\mu^*_{\eps_{t_0}}) \exp \lf( -\frac{c_1}{\kappa} (t^{\kappa} - t_0^{\kappa} ) \rh) \longrightarrow 0$ as $t \longrightarrow \infty$ for any finite $t_0$.
	The proof is then completed.
\end{proof}


\section{Convergence of the discrete-time simulated annealing}
\label{s:disSA}

	Following the sketch of proof in Section~\ref{s:proofsketch}, we provide in this section the convergence analysis of the discrete time kinetic simulated annealing process $(\Xb, \Yb)$ in \eqref{eq:diskl} and prove Theorem \ref{thm2}.

	Without loss of generality, we assume that the Lipschitz constant of $\gr U(x)$ satisfies $L\ge 1$.
	We also stay in the context with $\theta = 1$ for the proof.
	Denote $z = (x,y)$ and  $\Zb :=(\Xb,\Yb) $.
	By the explicit solution of $(\Xb_t,\Yb_t)$ in \eqref{eq:xbtybt}, we can see that the marginal distribution $\mub_t := \L(\Zb_t) $ is denoted by $\pb_t$ has a strictly positive and smooth density function $\pb_t(z)$.
	Also recall that $\mu^*_{\eps}(\d z) = p^*_{\eps}(z) dz$ with $p^*_{\eps}(x,y) \propto \exp\lf( -\frac{1}{\eps}(U(x)+\frac{|y|^2}{2}) \rh)\d z $.
	For the ease of notation, we abbreviate
	$$
		(\Xb_{T_k}, \Yb_{T_k},\Zb_{T_k}, \eps_{T_k}, \mub_{T_k}, \mu^*_{\eps_{T_k}}, \pb_{T_k}, p^*_{\eps_{T_k}} )
		~\;\mbox{to}\;~
		(\Xb_k, \Yb_k, \Zb_k, \eps_k, \mub_k, \mu^*_k, \pb_k, p^*_k ).
	$$

\subsection{Preliminary analysis of discrete kinetic annealing process}
\label{s:disprocess}

	We first provide a uniform boundedness result on the moment of $(\Xb, \Yb)$.
	To this end, we introduce two constants
	$$
		\e^* := \frac{ \min\lf\{ \frac{\tilde{\beta}r}{2L} , 1- \frac{\tilde{\beta}}{r}  \rh\} \min \lf\{1-\frac{\tilde{\beta}^2}{r/3} , \frac{1}{2}   \rh\}  }{ \max\lf\{ \frac{9L^2}{r} +\frac{4}{3} , L+6 \rh\} \max\lf\{ 1+\frac{3}{4r}, \frac{3}{2}  \rh\} },
		~~\mbox{ }~
		\tilde{\beta} :=\frac{1}{ 2 ( 1+\frac{1}{r}) \Big(1+\frac{1}{\sqrt{r/3}} \Big) },
	$$
	where $r > 0$ is the constant given in Assumption~\ref{assumption1}.
	One can check that
	$\tilde{\beta}\le \frac{1}{2}$, $\tilde{\beta} < r$ and $\frac{\tilde{\beta}^2}{r/3}<1$.

	\begin{prop} \label{prop:dismoment}
		Let Assumptions~\ref{assumption1} and \ref{assumption2} hold.
		Assume that $ \D t_k\le\min\lf\{ \frac{\eta^*}{2} , \frac{1}{L} \rh\} $ for sufficiently large $k \ge 0$.
		Then
		\begin{equation} \label{eq:Cb0}
			\Cb_0 ~:=~ \sup_{k \ge 0} \E\lf[U(\Xb_k)+ |\Yb_k|^2\rh]   < \infty.
		\end{equation}
	\end{prop}
	\begin{proof}
		Similar to the proof of Proposition~\ref{prop:conmoment}, we introduce the Lyapunov function $\tR(x,y)=U(x)+\frac{|y|^2}{2} + \tilde{\beta} x\cdot y $ for all $z = (x,y) \in \R^d \times \R^d$.
		Then, to prove \eqref{eq:Cb0},  it is equivalent to prove that $ \lf(\E\lf[\tR(\Xb_k, \Yb_k)\rh] \rh)_{k \ge 0}$ is uniformly bounded.

		Denote $ \theta_k := 1 - e^{-\D t_k} \le \D t_k$, so that the discrete time scheme \eqref{eq:implement} can be rewritten as
		\begin{equation*}
		\begin{cases}
			\Xb_{k+1} = \Xb_k + \theta_k \Yb_k - \lf( \D t_k - \theta_k \rh) \gr_x U(\Xb_k) + D_x(k), \\
			\Yb_{k+1}  = e^{-\D t_k} \Yb_k - \theta_k \gr_x U(\Xb_k) + D_y(k),
		\end{cases}
		\end{equation*}
		where $\ (D_x(k), D_y(k) ) $ is a Gaussian vector in $\R^d\times\R^d$ with mean zero and covariance matrix $\Sigma_k$ given by \eqref{eq:sigmak}.
		Then, by the Lipschitz property of $\gr_x U$, we derive
	\begin{eqnarray*}
		&& \E \lf[U(\Xb_{k+1}) \rh] \\
		& = & \E\lf[ U(\Xb_k) \rh] + \E\lf[\int_0^1 \gr_x U \lf( \Xb_k + t(\Xb_{k+1} - \Xb_k) \rh) \cdot (\Xb_{k+1} - \Xb_k) \d t\rh]\\
		& = & \E \lf[ U(\Xb_k) \rh]
		\E \lf[ \int_0^1 \gr_x U(\Xb_k) \cdot (\Xb_{k+1} - \Xb_k) \d t \rh]  \\
		&& + \E\lf[\int_0^1 \lf( \gr_x U \lf( \Xb_k + t(\Xb_{k+1} - \Xb_k) \rh) - \gr_x U(\Xb_k)  \rh) \cdot (\Xb_{k+1} - \Xb_k) \d t\rh]\\
		& \le & \E \lf[ U(\Xb_k) \rh] + \frac{L}{2} \E\lf[\lf|\Xb_{k+1} - \Xb_k\rh|^2 \rh] + \E\lf[\gr_x U(\Xb_k) \cdot (\Xb_{k+1}-\Xb_k)\rh]\\
		& = & \E \lf[ U(\Xb_k) \rh] + \frac{L}{2}\theta_k^2\E \big[ |\Yb_k|^2 \big] + \Big(\frac{L(\D t_k - \theta_k)^2}{2} - (\D t_k -\theta_k) \Big) \E\big[  \big|\gr_x U(\Xb_k)\big|^2 \big]\\
		&  & + \lf(\theta_k -L\theta_k (\D t_k - \theta_k)\rh) \E \lf[\Yb_k \cdot \gr_x U(\Xb_k)\rh] + \frac{L}{2}\Sigma_{11}(k) d.
	\end{eqnarray*}
	Further, we compute directly that
	\begin{eqnarray*}
		\E \lf[ \lf|\Yb_{k+1}\rh|^2 \rh] & = & \E \lf[ \lf|e^{-\D t_k} \Yb_k - \theta_k \gr_x U(\Xb_k) + D_y(k)\rh|^2\rh]\\
		& = & e^{-2\D t_k} \E\lf[\lf|\Yb_k\rh|^2 \rh] + \theta_k^2 \E\lf[\lf|\gr_x U(\Xb_k)\rh|^2 \rh] \\
		&&- 2\theta_k e^{-\D t_k} \E\lf[\Yb_k\cdot \gr_x U(\Xb_k)\rh] + \Sigma_{22}(k) d
	\end{eqnarray*}
	and
	\begin{eqnarray*}
		 \E \lf[\Xb_{k+1}\cdot \Yb_{k+1}\rh]
		\!\!\!\! &=&\!\!\!\! \E \Big[ ( \Xb_k + \theta_k \Yb_k - \lf( \D t_k - \theta_k \rh) \gr_x U(\Xb_k) \\
		&& \;\;\;\;\; + D_x(k) )\cdot \big( e^{-\D t_k} \Yb_k - \theta_k \gr_x U(\Xb_k) + D_y(k) \big) \Big]\\
		& = &\!\!\!\!  e^{-\D t_k} \E\lf[\Xb_k\cdot \Yb_k\rh] - \theta_k \E \lf[\Xb_k \cdot \gr_x U(\Xb_k)\rh] + \theta_k e^{-\D t_k} \E \lf[ \lf|\Yb_k\rh|^2 \rh] \\
		&  & - \lf( \theta_k^2 + (\D t_k -\theta_k)e^{-\D t_k}  \rh) \E\lf[\Yb_k \cdot \gr_x U(\Xb_k)\rh] \\
		&  & + \theta_k (\D t_k - \theta_k ) \E \lf[\lf|\gr_x U(\Xb_k)\rh|^2 \rh] + \Sigma_{12}(k) d.
	\end{eqnarray*}
	Therefore, by the definition of $\tR$, it follows that
\begin{eqnarray*}
      &  &  \frac{1}{\D t_k} \lf( \E\lf[\tR(\Xb_{k+1},\Yb_{k+1})\rh] - \E\lf[\tR (\Xb_k,\Yb_k ) \rh]\rh)\\
      & \le & \frac{1}{\D t_k} \bigg( \lf(\frac{L}{2}\theta_k^2 +\frac{1}{2} (e^{-2\D t_k}-1) + \tb\theta_k e^{-\D t_k}  \rh)\E\lf[\lf|\Yb_k\rh|^2\rh] + \tb\lf( e^{-\D t_k} -1 \rh) \E\lf[\Xb_k\cdot \Yb_k\rh]\\
      &  & + \lf( \frac{L(\D t_k-\theta_k)^2}{2} - (\D t_k -\theta_k) + \frac{1}{2}\theta_k^2 + \tb\theta_k(\D t_k-\theta_k)  \rh) \E\lf[\lf|\gr_x U(\Xb_k)\rh|^2\rh] \\
      &  & + \lf(\theta_k -L\theta_k(\D t_k-\theta_k) -\theta_k e^{-\D t_k} - \tb \lf( \theta_k^2 + (\D t_k - \theta_k)e^{-\D t_k} \rh) \rh) \E\lf[\Yb_k \cdot \gr_x U(\Xb_k)\rh]\\
      &  & - \tb\theta_k\E\lf[\Xb_k\cdot \gr_x U(\Xb_k)\rh] + \frac{L}{2}\Sigma_{11}(k)d + \frac{1}{2}\Sigma_{22}(k)d + \tb \Sigma_{12}(k) d\bigg) .
\end{eqnarray*}
Using the fact that $\tb\le \frac{1}{2}$, $\theta_k\le \D t_k \le \frac{1}{L} \le 1 $, and the expressions of $\Sigma_{ij}(k)$ in \eqref{eq:sigmak}, we have
$$\frac{L}{2}\theta_k^2 +\frac{1}{2} (e^{-2\D t_k}-1) + \tb\theta_k e^{-\D t_k} \le -\frac{1}{2}\D t_k + \lf( \frac{L}{2}+ 2 -\frac{3}{2} \tb \rh)\D t_k^2 ,$$
$$ \frac{L(\D t_k-\theta_k)^2}{2} - (\D t_k -\theta_k) + \frac{1}{2}\theta_k^2 + \tb\theta_k(\D t_k-\theta_k) \le \frac{1}{6}\D t_k^3 + \frac{L}{8} \D t_k^4  ,$$
$$\theta_k -L\theta_k(\D t_k-\theta_k) -\theta_k e^{-\D t_k} - \tb \lf( \theta_k^2 + (\D t_k - \theta_k)e^{-\D t_k} \rh) \le \lf( 1 - \frac{3}{2}\tb \rh)\D t_k^2 + \lf( 2\tb - \frac{L}{2} \rh)\D t_k^3   ,$$
and
$$ \frac{L}{2}\Sigma_{11}(k)d + \frac{1}{2}\Sigma_{22}(k)d + \tb \Sigma_{12}(k) d \le \frac{5}{2} \eps_k\D t_k d  .$$
	Thus, by $(r,m)$-dissipativity of $U$ and the $L$-Lipschitz of $\gr_x U$, there exists some constant $b$ such that
	\begin{eqnarray*}
        &  &  \frac{1}{\D t_k} \lf( \E\lf[\tR(\Xb_{k+1},\Yb_{k+1})\rh] - \E\lf[\tR(\Xb_k,\Yb_k) \rh]\rh)\\
        & \le & \frac{1}{\D t_k} \bigg\{ \lf( -\frac{1}{2}\D t_k + \lf( \frac{L}{2}+ 2 - \frac{3}{2} \tb \rh)\D t_k^2 \rh) \E\lf[ \lf|\Yb_k\rh|^2\rh] + \lf( \frac{1}{6}\D t_k^3 + \frac{L}{8} \D t_k^4 \rh) \E\lf[\lf|\gr_x U(\Xb_k)\rh|^2\rh] \\
        &  & \lf( \lf( 1 - \frac{3}{2}\tb \rh)\D t_k^2 + \lf( 2\tb - \frac{L}{2} \rh)\D t_k^3 \rh) \frac{\E\lf[\lf|\Yb_k\rh|^2\rh] + \E\lf[\lf|\gr_x U(\Xb_k)\rh|^2\rh] }{2}  \\
        &  & + \tb \lf( -\D t_k + \frac{\D t_k^2}{2} \rh) \E\lf[\Xb_k \cdot \Yb_k\rh] -\tb \lf( \D t_k -\frac{1}{2}\D t_k^2 \rh) \lf(r\E\lf[\lf|\Xb_k\rh|^2\rh]-m\rh) + \frac{5}{2}\eps_k\D t_k d \bigg\}\\
        & \le &  -\tb r\E\lf[\lf|\Xb_k\rh|^2\rh] - \frac{1}{2} \E\lf[\lf|\Yb_k\rh|^2\rh] - \tb\lf( 1-\frac{1}{2}\D t_k \rh) \E\lf[\Xb_k\cdot \Yb_k\rh] \\
        &  &  + \lf( \lf( L^2 + \frac{1}{2}\tb r \rh)\D t_k + L^2 \lf(2\tb +1 \rh)\D t_k^2  \rh) \E\lf[\lf|\Xb_k\rh|^2\rh]\\
        & & + \lf( \frac{L+5}{2} \D t_k + \tb \D t_k^2 \rh) \E \lf[ \lf| \Yb_k \rh|^2\rh] + b  \\
        & \le & -\tb r\E\lf[\lf|\Xb_k\rh|^2\rh] - \frac{1}{2} \E\lf[\lf|\Yb_k\rh|^2\rh] + \tb \lf(c\E\lf[\lf|\Xb_k\rh|^2\rh] + \frac{1}{4c} \E\lf[\lf|\Yb_k\rh|^2\rh]  \rh)\\
        &   & + \D t_k \lf( \frac{L+6}{2}\E\lf[\lf|\Yb_k\rh|^2\rh] + \lf( 3L^2 + \frac{r}{4} \rh)\E\lf[\lf|\Xb_k\rh|^2\rh] \rh) + b
\end{eqnarray*}
for arbitrary $c$. Now choose $c=\frac{r}{2}$.
	Notice that $\tb < r$ by its definition, then
	\begin{eqnarray*}
		&  &\frac{1}{\D t_k} \lf( \E\lf[\tR(\Xb_{k+1},\Yb_{k+1})\rh] - \E\lf[\tR(\Xb_k,\Yb_k) \rh]\rh)\\
		& \le & -\frac{1}{2} \tb r \E\lf[\lf|\Xb_k\rh|^2\rh] - \frac{1}{2}\lf( 1- \frac{\tb}{r} \rh) \E\lf[\lf|\Yb_k\rh|^2\rh]\\
		&  & + \D t_k \lf( \frac{L+6}{2}\E\lf[\lf|\Yb_k\rh|^2\rh] + \lf( 3L^2 + \frac{r}{4} \rh) \E\lf[\lf|\Xb_k\rh|^2\rh] \rh) + b.
	\end{eqnarray*}
	Using \eqref{quadratic}, we derive
	\begin{eqnarray*}
		&  &\frac{1}{\D t_k} \lf( \E\lf[\tR(\Xb_{k+1},\Yb_{k+1})\rh] - \E\lf[\tR(\Xb_k,\Yb_k) \rh]\rh)\\
		& \le & -\frac{\tb r}{2L}\E\lf[ U(\Xb_k) \rh] - \frac{1}{2} \lf( 1 - \frac{\tb}{r} \rh) \E\lf[\lf|\Yb_k\rh|^2\rh]\\
		&  & + \D t_k \lf( \lf(\frac{9L^2}{r} + \frac{3}{4} \rh) \E\lf[ U(\Xb_k)\rh] + \frac{L+6}{2} \E\lf[\lf|\Yb_k\rh|^2\rh] \rh) + b_3 \\
		& \le & -b_1 \lf(\E\lf[ U(\Xb_k) \rh] + \frac{1}{2} \E\lf[ \lf|\Yb_k\rh|^2 \rh]  \rh) + b_2\D t_k \lf(\E\lf[ U(\Xb_k)\rh] + \frac{1}{2} \E \lf[ |\Yb_k|^2 \rh]  \rh) + b_3,
	\end{eqnarray*}
	where $b_1 :=\min\lf\{ \frac{\tb r}{2L} , 1-\frac{\tb}{r}  \rh\}  $, $b_2 :=\max\lf\{ \frac{9L^2}{r}+\frac{4}{3}, L+6  \rh\}  $, and $b_3 := b + \frac{\tb r K}{2L} + \lf( \frac{9L^2}{r}+\frac{4}{3}\rh) K$.
	Set $b_4 :=\max\lf\{1+ \frac{3}{4r}, \frac{3}{2}  \rh\}  $ and $b_5 :=\min\lf\{ \frac{1}{2} , 1-\frac{\tb^3}{r/3}  \rh\}>0  $. Then, we have
	\begin{eqnarray*}
		b_5 \lf( U(x) + \frac{1}{2}|y|^2 \rh) - \frac{3\tb^2 K}{r}
		\le
		\tR(x,y)
		\le
		b_4 \lf( U(x) + \frac{1}{2}|y|^2 \rh) + \frac{3K}{4r}.
	\end{eqnarray*}
	We thus obtain
	\begin{eqnarray*}
		&  &\frac{1}{\D t_k} \lf( \E\lf[\tR(\Xb_{k+1},\Yb_{k+1})\rh] - \E\lf[\tR(\Xb_k,\Yb_k) \rh]\rh)\\
		& \le & -\frac{b_1}{b_4} \lf( \E\lf[\tR(\Xb_k,\Yb_k) \rh] - \frac{3K}{4r} \rh) + \D t_k \frac{b_2}{b_5} \lf( \E\lf[\tR(\Xb_k, \Yb_k) \rh] + \frac{3\tb^2 K}{r} \rh) + b_3.
	\end{eqnarray*}
	Rearrange the above, we derive, for some constant $C > 0$, that
	\begin{equation*}
		\E\lf[\tR(\Xb_{k+1}, \Yb_{k+1}) \rh] \le \lf(1 - \frac{b_1}{b_4}\D t_k + \frac{b_2}{b_5}\D t_k^2 \rh)\E\lf[\tR(\Xb_k,\Yb_k)\rh] + C\D t_k.
	\end{equation*}
	When  $\D t_k\le \frac{\e^*}{2} $, we have $ 1 - \frac{b_1}{b_4}\D t_k + \frac{b_2}{b_5}\D t_k^2 \le 1 - \frac{b_1}{2 b_4} \D t_k  $,
	which implies that $\big( \E\lf[\tR(\Xb_k,\Yb_k)\rh] \big)_{k \ge 0}$ is uniformly bounded.
	\end{proof}

	\begin{rmk}
		Under the conditions in Theorem~\ref{thm2}, we have
		\begin{equation*}
			\lim_{k \to \infty} T_k = \infty
			\;\;\mbox{and}\;\;
			\lim_{k \to \infty}  \D t_k = 0,
		\end{equation*}
		so the conditions of Proposition~\ref{prop:dismoment} hold.
	\end{rmk}

	By Proposition~\ref{prop:dismoment}, we can easily obtain that $\big( \E\big[ |\Yb_t |^2 \big] \big)_{t \ge 0}$ is also uniformly bounded.
	Moreover, we have the following fine estimation on the increment of $\Xb$:
	
	\begin{prop} \label{prop:disboundx}
		There exists a constant $C > 0$, such that
		$$
			\sup_{t \in [T_k, T_{k+1}]}
			\E\lf[\lf|\Xb_t-\Xb_k\rh|^2\rh]\le C \D t_k^2,
			~~\mbox{for all}~
			k \ge 0.
		$$
	\end{prop}

	\begin{proof}
	For $t \in [T_k, T_{k+1}]$, notice that $\Xb_t=\Xb_k + \int_{T_k}^t \Yb_s \d s $.
	It then follows by Cauchy-Schwarz inequality that
\begin{eqnarray*}
        \E\lf[\lf|\Xb_t-\Xb_k\rh|^2\rh]
        = \E \lf[\lf| \int_{T_k}^t \Yb_s \d s \rh|^2 \rh]
        \le (t-T_k) \int_{T_k}^t \E\lf[\lf|\Yb_s\rh|^2\rh] \d s
        \le C\D t_k^2
\end{eqnarray*}
	for some constant $C$ independent of $k \ge 0$.
\end{proof}


\subsection{One-step entropy evolution with fixed temperature}
\label{s:disdissipation}
	In this subsection, we aim at estimating the difference term \eqref{eq:derdis1},
	which is the one-step evolution of $H_{\gamma(\eps)}\lf( \mub_t | \mu^*_{\eps} \rh)$ with fixed $\eps$.

	We first provide the Fokker-Planck equation satisfied by $\mub_t$ on each interval $[T_k, T_{k+1}]$.
	Recall that $\pb_t$ denotes the density function of $\mub_t$.
	For $t \in [T_k, T_{k+1}]$ and $z_k \in \R^d \times \R^d$, denote by
	$$
		z \longmapsto \pb_{t|T_k}\lf(z| z_k\rh)
		~\mbox{ the density function of }~
		\L (\Zb_t | \Zb_{T_k} = z_k),
	$$
	that is, the conditional distribution of $\Zb_t$ knowing $\Zb_{T_k} = z_k$.
	Further, similar to \eqref{vel1}, define
	$$
		v_{t,k} (z) := \begin{pmatrix}
		y \\
		-\gr_x U(x) - y -\eps_k\gr_y \log \pb_t
		\end{pmatrix}
		~\mbox{and}~
		\hat{v}_{t,k}(z)
		:= \begin{pmatrix}
	        y \\
		-\gr_x U(x_k) - y -\eps_k\gr_y \log \pb_t
		\end{pmatrix}.
	$$
	
	\begin{lem} \label{lem:disFP}
		For all $k \ge 0$, $\pb_t$ satisfies
	$$
		 \dr_t \pb_t = - \gr_z \cdot \lf( \pb_t v_{t,k}  \rh) - \int \lf[ \gr_z \cdot \lf( \pb_{t|T_k}\lf(\cdot|z_k\rh) \lf(\hat{v}_{t,k} - v_{t,k}  \rh) \rh)  \rh] \d \mub_k(z_k),
		~
		 t\in[T_k ,T_{k+1}].
	$$
\end{lem}

\begin{proof}
	Conditioned on the $\sigma$-field $\overline{\F}_{T_k} :=\sigma (\Zb_t : 0 \le t \le T_k ) $,
	 $\Zb$ follows a linear SDE with deterministic parameters on $[T_k, T_{k+1}]$ (see \eqref{eq:diskl}).
	Thus, the conditional density function $z \longrightarrow p(z\;|\;z_k )$ satisfies the Fokker-Planck equation
	\begin{eqnarray*}
		\dr_t \pb_{t|T_k}(z|z_k) & = &  -\gr_z \cdot \lf(\pb_{t|T_k}(z|z_k) \hat{v}_{t,k}(z) \rh)\\
		& = & -\gr_z \cdot \lf(\pb_{t|T_k}(z|z_k) v_{t,k}(z) \rh) - \gr_z \cdot \lf(\pb_{t|T_k}(z|z_k) (\hat{v}_{t,k}(z) - v_{t,k}(z) ) \rh).
	\end{eqnarray*}
	Notice that $v_{t,k}$ is independent of $z_k$. We can then complete the proof by integrating both sides of the above equality with respect to $z_k$ under the measure $\mub_k(z_k)$.
\end{proof}


For $t\in[T_k, T_{k+1}] $, we write the distorted relative entropy
\begin{eqnarray*}
    H_{\gamma(\eps_k)}(\mub_t|\mu^*_k) & = & \;\E_{\mub_t} \lf[ \lf\langle \gr_z\log\frac{\pb_t}{p^*_k}, S\,\gr_z\log \frac{\pb_t}{p^*_k} \rh\rangle + \gamma(\eps_k)\log \frac{p_t}{p^*_k} \rh] \no \\
     & = & 4\;\;\E_{\mub_t}\lf[\lf\langle \gr_z\log \th_t, S\,\gr_z\log \th_t \rh\rangle    \rh] + 2\gamma(\eps_k)\;\;\E_{\mub_t}\lf[\log \th_t   \rh],
\end{eqnarray*}
where $\tilde{h}_t=\sqrt{\frac{\pb_t}{p^*_k}}$ and $S=\lf(\begin{array}{cc}
    I_d & I_d  \\
    I_d & I_d
\end{array}\rh)$.
	Then, for $t\in[T_k ,T_{k+1}]$, by Lemma~\ref{lem:disFP}, we have
	\begin{eqnarray}
		&& \frac{\d}{\d t} H_{\gamma(\eps_k)}(\mub_t|\mu^*_k)
		~=~
		\int \frac{\delta H_{\gamma(\eps_k)}(\mub_t|\mu^*_k)}{\delta\; \mub_t} \;\; \dr_t \pb_t \; \d z \no \\
		\;\; &=&\!\!\!\! \label{a1} \int \lf<\gr_z \lf( \frac{\delta H_{\gamma(\eps_k)}(\mub_t|\mu^*_k)}{\delta\; \mub_t} \rh),\; v_{t,k}  \rh>\; \pb_t \;\d z \\
		& + &   \label{b1}  \int \lf<\gr_z \lf( \frac{\delta H_{\gamma(\eps_k)}(\mub_t|\mu^*_k)}{\delta\; \mub_t} \rh),\; \int \pb_{t|T_k}(z|z_k) (\hat{v}_{t,k}- v_{t,k} )(z)\; \pb_k(z_k) \d z_k  \rh>\; \d z.
	\end{eqnarray}
	In the above, the first order variational derivative of $ H_{\gamma(\eps_k)}(\mub_t|\mu^*_k)$ is given by
	$$\frac{\delta H_{\gamma(\eps_k)}(\mub_t|\mu^*_k)}{\delta\; \mub_t} = \frac{4}{\th_t}\tgr_z^* S\gr_z\th_t + 2\gamma(\eps_k) \log\th_t + \gamma(\eps_k) ,$$
	where we use $(\tilde{\gr}_z)^*=(\tilde{\gr}_x^*, \tilde{\gr}_y^*)\; := \; \lf(-\gr_x\cdot + \frac{1}{\eps_k}\gr_x U\cdot\, , \, -\gr_y\cdot + \frac{y}{\eps_k} \cdot\rh)$ to denote the adjoint operator of $\gr_z$ with respect to $\mu^*_k$, in the sense that for any $f,g \in L^2(\mu^*_k)$, $\;\E_{\mu^*_k}[f\cdot \gr_z g] = \;\E_{\mu^*_k}[g\tilde{\gr}_z^* f] $.

    We then follow the same computation steps as in \eqref{ex1}, \eqref{ex2} and \eqref{ex3} to obtain that the term \eqref{a1} is equal to
	\begin{equation}\label{a2}
	    -4 \;\E_{\mub_t} \lf[\lf\langle \gr_z \log \th_t, M_k \gr_z \log \th_t \rh\rangle  \rh] -8 \eps_k  \;\E_{\mub_t} \lf[ \lf\langle
		\gr_z\gr_y \log \th_t , S \gr_z\gr_y \log \th_t
		\rh\rangle_F    \rh],
	\end{equation}
	where $M_k=\begin{pmatrix}
		2I_d & 2I_d -\gr_x^2 U\\
		2I_d - \gr_x^2 U & \lf( 2+ \gamma(\eps_k)\eps_k  \rh)I_d - 2\gr_x^2 U
		\end{pmatrix} \in\M_{2d\times 2d}(\R) $.

	Because $\hat{v}_{t,k}(z) -v_{t,k}(z) = \left( \begin{array}{c}
        0 \\
        \gr_x U(x) -\gr_x U(x_k)
        \end{array}\right),$ term \eqref{b1} can be directly computed as
\begin{eqnarray}
        &  &  \int \Big\langle \gr_z \lf( \frac{4}{\th_t} \tilde{\gr}_z^* S\gr_z \th_t \rh) + 2\gamma(\eps_k)\gr_z \log\th_t , \no \\
        	&&  \;\;\;\;\;\;\;\;\;\;\;\; \int  (\hat{v}_{t,k}- v_{t,k} )(z)\; \pb_{t|T_k}(z|z_k) \; \pb_k(z_k)\d z_k \Big\rangle \d z  \\
        & = & \int \lf\langle 2\gamma(\eps_k) \gr_y\log\th_t, \; \;\int  \lf(\gr_x U(x) -\gr_x U(x_k) \rh)\; \pb_{t|T_k}(z|z_k) \; \pb_k(z_k) \d z_k \rh\rangle\d z \no\\
        &  & +  \int \lf\langle \gr_y\lf( \frac{4}{\th_t} \tilde{\gr}_z^* S\gr_z \th_t \rh) , \; \int (\gr_x U(x) -\gr_x U(x_k) ) \; \pb_{t|T_k}(z|z_k)\;  \pb_k(z_k)\d z_k\rh\rangle \d z \no.
\end{eqnarray}

Denote
\begin{equation}
    \label{eq:A}
A_t(z) := \frac{1}{\pb_t(z)} \;\int \lf(\gr_x U(x) -\gr_x U(x_k) \rh) \; \pb_{t|T_k}(z|z_k) \; \pb_k(z_k)\d z_k.
\end{equation}
Because $\tgr_z^* S\gr_z\th_t = \tgr_x^*\gr_x \th_t + \tgr_x^*\gr_y \th_t + \tgr_y^*\gr_x \th_t + \tgr_y^*\gr_y \th_t $, in the same spirit of  \cite[Lemma 3]{MaBEJ21}, we have the following results:

\begin{lem}
     \label{lemb2}
For $\# \in\{ x,y \} $,
\begin{eqnarray*}
       & & 4\int \lf\langle \gr_y \lf(\frac{\tilde{\gr}_x^*\gr_\#\th_t}{\th_t}  \rh), \; \int \lf(\gr_x U(x) -\gr_x U(x_k) \rh) \; \pb_{t|T_k}(z|z_k) \; \pb_k(z_k)\;\d z_k  \rh\rangle \d z\\
       & = & 4\int\lf\langle \gr_y\gr_{\#}\log\th_t, \gr_x A_t(z)  \rh\rangle_F \d \mub_t(z) + 4 \int \lf\langle \gr_y\gr_{\#}\th_t, A_t(z)(\gr_x\th_t)^{\T}  \rh\rangle_F \d \mu^*_k(z)\\
       & & - 4\int \lf\langle \gr_x\gr_y\th_t, A_t(z)(\gr_{\#}\th_t)^{\T}  \rh\rangle_F \d \mu^*_k(z),
\end{eqnarray*}
and
\begin{eqnarray*}
       & & 4\int \lf\langle \gr_y \lf(\frac{\tilde{\gr}_y^*\gr_\#\th_t}{\th_t}  \rh), \; \int \lf(\gr_x U(x) -\gr_x U(x_k) \rh)\; \pb_{t|T_k}(z|z_k) \; \pb_k(z_k)\d z_k  \rh\rangle \d z\\
       & = & 4\int\lf\langle \gr_y\gr_{\#}\log\th_t, \gr_y A_t(z) \rh\rangle_F \d \mub_t(z) + 4 \int \lf\langle \gr_y\gr_{\#}\th_t, A_t(z)(\gr_y\th_t)^{\T}  \rh\rangle_F \d \mu^*_k(z)\\
       & &  + \frac{4}{\eps_k} \int \lf\langle \gr_{\#} \log\th_t, \; \int  \lf(\gr_x U(x) -\gr_x U(x_k) \rh) \; \pb_{t|T_k}(z|z_k) \; \pb_k(z_k)\d z_k \rh\rangle \d z\\
       &  & - 4\int \lf\langle \gr_y^2\th_t, A_t(z)(\gr_{\#}\th_t)^{\T}  \rh\rangle_F \d \mu^*_k(z).
\end{eqnarray*}
Consequently, \eqref{b1} is equal to
\begin{eqnarray*}
        & & 4\int \lf\langle \gr_z\gr_y\log \th_t, S \gr_z A_t(z)  \rh\rangle_F \d \mub_t (z) + \int\bigg\langle \frac{4}{\eps_k}\gr_x\log \th_t + \lf(2\gamma(\eps_k) +\frac{4}{\eps_k}  \rh)\gr_y\log\th_t,\\
        &  &  \int \lf(\gr_x U(x) -\gr_x U(x_k) \rh)\; \pb_{t|T_k}(z|z_k) \; \pb_k(z_k)\d z_k \bigg\rangle\d z.
\end{eqnarray*}
\end{lem}

%
%

%
%

\vspace{0.5em}
	Using with \eqref{a2} and Lemma~\ref{lemb2}, we derive that
	\begin{eqnarray} \label{eq:dHk_inter}
		&&\frac{\d}{\d t} H_{\gamma(\eps_k)}(\mub_t | \mu^*_k)  \\
		 & = & -4\;\E_{\mub_t} \lf[\lf\langle \gr_z \log\th_t, M_k\gr_z\log\th_t \rh\rangle \rh]  -8\eps_k \;\E_{\mub_t} \lf[\lf\langle \gr_z\gr_y\log\th_t, S\gr_z\gr_y\log\th_t  \rh\rangle_F  \rh] \no \nonumber\\
		&  & + 4\int \lf\langle \gr_z\gr_y\log\th_t, S\gr_z A_t(z) \rh\rangle_F \d\mub_t (z) \nonumber \\
		&& + \int\bigg\langle \frac{4}{\eps_k}\gr_x\log \th_t + \lf(2\gamma(\eps_k) +\frac{4}{\eps_k}  \rh)\gr_y\log\th_t, \nonumber \\
		&&\;\;\;\;\; \;\;\;\;\; \;\;\;\;\;
			\int \lf(\gr_x U(x) -\gr_x U(x_k) \rh)\; \pb_{t|T_k}(z|z_k) \; \pb_k(z_k)\d z_k \bigg\rangle \d z. \nonumber
	\end{eqnarray}

\begin{lem}
\label{lem:young}
	It holds that
\begin{eqnarray*}
        &  & \int\bigg\langle \frac{4}{\eps_k}\gr_x\log \th_t + \lf(2\gamma(\eps_k) +\frac{4}{\eps_k}  \rh)\gr_y\log\th_t, \\
		&&\;\;\;\;\; \;\;\;\;\; \;\;\;\;\;
			\int \lf(\gr_x U(x) -\gr_x U(x_k) \rh)\; \pb_{t|T_k}(z|z_k) \; \pb_k(z_k)\d z_k \bigg\rangle \d z\\
        & \le & \;\E_{\mub_t} \lf[\lf|\gr_x\log\th_t\rh|^2\rh] + \lf( 1+ \frac{1}{2} \eps_k\gamma(\eps_k) \rh) \;\E_{\mub_t} \lf[\lf| \gr_y\log\th_t \rh|^2\rh]\\
        & & + \frac{4L^2}{\eps_k^2} \lf( 2+ \frac{1}{2}\eps_k\gamma(\eps_k) \rh) \;\E  \lf[\lf|\Xb_t - \Xb_k\rh|^2\rh].
\end{eqnarray*}
\end{lem}

\begin{proof}
Using $a\cdot b\le \frac{\eps_k}{4} |a|^2 + \frac{1}{\eps_k}|b|^2 $ and $\pb_{t|T_k} (z|z_k)\pb_k(z_k) = \pb_t(z) \pb_{T_k|t}(z_k | z) $ where the backward conditional density function $\pb_{T_k|t}(z_k | z) := p(\Zb_k = z_k | \Zb_t =z ) $, and by Jensen's inequality, we can show that
\begin{eqnarray*}
        &  & \frac{4}{\eps_k} \int\lf\langle \gr_x\log \th_t,  \; \int (\gr_x U(x) - \gr_x U(x_k) ) \; \pb_{t|T_k}(z|z_k) \; \pb_k(z_k)\d z_k \rh\rangle\d z\\
        & = & \frac{4}{\eps_k} \int\lf\langle \gr_x\log \th_t,  \; \int ( \gr_x U(x) - \gr_x U(x_k) )\pb_{T_k|t}(z_k | z) \d z_k  \rh\rangle \pb_t(z) \d z \\
        & \le & \;\E_{\mub_t} \lf[\lf|\gr_x\log\th_t\rh|^2\rh] + \frac{4}{\eps_k^2} \;\E \lf[\lf|\gr_x U(\Xb_t) - \gr_x U(\Xb_k)\rh|^2\rh]\\
        & \le & \;\E_{\mub_t} \lf[\lf|\gr_x\log\th_t\rh|^2\rh] + \frac{4L^2}{\eps_k^2} \;\E \lf[\lf|\Xb_t-\Xb_k\rh|^2\rh].
\end{eqnarray*}
The $\gr_y\log\th_t$ term can be treated similarly. The proof then completes.
\end{proof}

\begin{lem}
\label{lem:gradA}
Assume that $\D t_k\le 1/L$ for every $k$ and $\gr_x^2 U $ is $L'$-Lipschitz. Then
\begin{eqnarray}\label{eq:gradientA}
        &  & 4 \int \lf\langle\gr_z\gr_y\log\th_t, S\gr_z A_t(z)\rh\rangle_F \d\mub_t(z) \\
        & \le & 8\eps_k\;\E_{\mub_t} \lf[ \lf\langle\gr_z\gr_y \log\th_t, S\gr_z\gr_y \log\th_t \rh\rangle_F \rh]  + \frac{(L')^2}{\eps_k} \;\E \lf[\lf|\Xb_t-\Xb_k\rh|^2\rh] + \frac{144}{\eps_k} d^{3/2} L^2 \D t_k^2.\no
\end{eqnarray}
\end{lem}

\begin{proof}
Without loss of generality, we assume $L\ge 1$. Using the same proof as the one of \cite[Lemma 10]{MaBEJ21}, we can show that
\begin{eqnarray}
\label{eq:gradA}
    \gr_z A_t(z) = \int \lf( \begin{array}{c}
        \gr_x^2 U(x) - \gr_x^2 U(x_k) - \gr_x^2 U(x_k)\lf( (I_d +\varpi(t)\gr_x^2 U(x_k))^{-1} -I_d \rh)   \\
        \lf(e^{t-T_k} -1 \rh) \gr_x^2 U(x_k) (I_d +\varpi(t)\gr_x^2 U(x_k))^{-1}
        \end{array}\rh) \nonumber \\
         \pb_{T_k|t}(z_k|z)\d z_k,
\end{eqnarray}
where $\varpi(t)=e^{t-T_k} - 1 - (t-T_k) $. As a result,
\begin{eqnarray*}
        &  &  4 \int \lf\langle\gr_z\gr_y\log\th_t, S\gr_z A_t(z)\rh\rangle_F \d\mub_t(z)\\
        & = & 4 \E\lf[ \lf\langle \gr_z\gr_y \log\th_t, S \lf( \begin{array}{c}
        \gr_x^2 U(\Xb_t) - \gr_x^2 U(\Xb_k) \\
        0
        \end{array}\rh) \rh\rangle_F \rh]\\
        &  & + 4 \E\lf[ \lf\langle \gr_z\gr_y \log\th_t, S \lf( \begin{array}{c}
         - \gr_x^2 U(\Xb_k)\lf( (I_d +\varpi(t)\gr_x^2 U(\Xb_k))^{-1} -I_d \rh)   \\
        \lf(e^{t-T_k} -1 \rh) \gr_x^2 U(\Xb_k) (I_d +\varpi(t)\gr_x^2 U(\Xb_k))^{-1}
        \end{array}\rh)  \rh\rangle_F \rh].
\end{eqnarray*}
On the one hand,
\begin{eqnarray*}
        &  & 4 \E\lf[ \lf\langle \gr_z\gr_y \log\th_t, S \lf( \begin{array}{c}
        \gr_x^2 U(\Xb_t) - \gr_x^2 U(\Xb_k) \\
        0
        \end{array}\rh) \rh\rangle_F \rh]\\
        & \le & 4 \eps_k \int \lf\langle \gr_z\gr_y \log\th_t,  S\gr_z\gr_y \log\th_t  \rh\rangle_F \d \mub_t(z) \\
        &  & + \frac{1}{\eps_k} \E\lf[ \lf\langle \lf( \begin{array}{c}
        \gr_x^2 U(\Xb_t) - \gr_x^2 U(\Xb_k) \\
        0
        \end{array}\rh), S \lf( \begin{array}{c}
        \gr_x^2 U(\Xb_t) - \gr_x^2 U(\Xb_k) \\
        0
        \end{array}\rh)   \rh\rangle_F \rh] \\
        & = & 4\eps_k\;\E_{\mub_t} \lf[ \lf\langle\gr_z\gr_y \log\th_t, S\gr_z\gr_y \log\th_t \rh\rangle_F \rh]  + \frac{1}{\eps_k} \;\E \lf[ \lf\|\gr_x^2 U(\Xb_t) -\gr_x^2 U(\Xb_k) \rh\|^2_F  \rh]\\
        & \le & 4\eps_k\;\E_{\mub_t} \lf[ \lf\langle\gr_z\gr_y \log\th_t, S\gr_z\gr_y \log\th_t \rh\rangle_F \rh] + \frac{(L')^2}{\eps_k} \;\E \lf[\lf|\Xb_t-\Xb_k\rh|^2\rh],
\end{eqnarray*}
where the last inequality is the case due to the $L'$-Lipschitz of $\gr_x^2 U$. On the other hand, denoting $J_t(\Xb_k):=\lf( \begin{array}{c}
         - \gr_x^2 U(\Xb_k)\lf( (I_d +\varpi(t)\gr_x^2 U(\Xb_k))^{-1} -I_d \rh)   \\
        \lf(e^{t-T_k} -1 \rh) \gr_x^2 U(\Xb_k) (I_d +\varpi(t)\gr_x^2 U(\Xb_k))^{-1}
        \end{array}\rh)$, we have
\begin{eqnarray*}
        &  & 4\; \E\lf[ \lf\langle \gr_z\gr_y \log\th_t, S \lf( \begin{array}{c}
         - \gr_x^2 U(\Xb_k)\lf( (I_d +\varpi(t)\gr_x^2 U(\Xb_k))^{-1} -I_d \rh)   \\
        \lf(e^{t-T_k} -1 \rh) \gr_x^2 U(\Xb_k) (I_d +\varpi(t)\gr_x^2 U(\Xb_k))^{-1}
        \end{array}\rh)  \rh\rangle_F \rh]\\
        & \le & 4\eps_k\;\;\E_{\mub_t} \lf[ \lf\langle\gr_z\gr_y \log\th_t, S\gr_z\gr_y \log\th_t \rh\rangle_F \rh] + \frac{1}{\eps_k}\; \E_{\mub_k} \lf[ \lf\langle J_t(\Xb_k), S J_t(\Xb_k) \rh\rangle_F \rh].
\end{eqnarray*}
As a result, we derive,
\begin{eqnarray*}
        &  & 4 \int \lf\langle\gr_z\gr_y\log\th_t, S\gr_z A_t(z)\rh\rangle_F \d\mub_t(z)\\
        & \le & 8\eps_k\;\;\E_{\mub_t} \lf[ \lf\langle\gr_z\gr_y \log\th_t, S\gr_z\gr_y \log\th_t \rh\rangle_F \rh] + \frac{(L')^2}{\eps_k} \;\E \lf[\lf|\Xb_t-\Xb_k\rh|^2\rh]\\
        &  & +\frac{1}{\eps_k}\; \E_{\mub_k} \lf[ \lf\langle J_t(\Xb_k), S J_t(\Xb_k) \rh\rangle_F \rh].
\end{eqnarray*}
Comparing the above inequality with \eqref{eq:gradientA}, to complete the proof of the lemma, we only need to prove
\begin{equation}
\label{eq:boundJ1}
    \E_{\mub_k} \lf[ \lf\langle J_t(\Xb_k), S J_t(\Xb_k) \rh\rangle_F \rh] \le 144 d^{3/2} L^2 \D t_k^2.
\end{equation}

By Cauchy-Schwarz inequality and matrix norm inequality, we have
\begin{eqnarray}
        \E_{\mub_k} \lf[ \lf\langle J_t(\Xb_k), S J_t(\Xb_k) \rh\rangle_F \rh] & \le & \;\E_{\mub_k} \lf[\lf\| J_t(\Xb_k) \rh\|_F \lf\| SJ_t(\Xb_k)\rh\|_F  \rh]\no\\
        & \le & \;\E_{\mub_k} \lf[ \|S\|_F \cdot\rank \lf(J_t(\Xb_k)\rh) \lf\|J_t(\Xb_k)\rh\|^2_2 \rh]\no\\
        & \le & \label{eq:boundJ2} 2d^{\frac{3}{2}} \;\E_{\mub_k} \lf[ \lf\| J_t(\Xb_k) \rh\|^2_2 \rh].
\end{eqnarray}
It is straightforward to see that
\begin{eqnarray*}
        \lf\|J_t(\Xb_k)\rh\|_2 & \le & \sqrt{2} \max\Big\{\lf\| \gr_x^2 U(\Xb_k) \lf( (I_d +\varpi(t)\gr_x^2 U(\Xb_k))^{-1} -I_d \rh) \rh\|_2, \\
        &  & \lf\| \lf(e^{t-T_k} -1 \rh) \gr_x^2 U(\Xb_k) (I_d +\varpi(t)\gr_x^2 U(\Xb_k))^{-1} \rh\|_2 \Big\}  .
\end{eqnarray*}
Because $\lf\|\gr_x^2 U \rh\|_2\le L $, which is due to the $L$-Lipschitz of $\gr_x U$, and because $\D t_k\le 1/L $, we have $\varpi (t) \le \D t_k^2 $ and $\lf(I_d +\varpi(t)\gr_x^2 U(\Xb_k)\rh)^{-1}\preceq I_d -\varpi(t)\gr_x^2 U(\Xb_k) +\lf( \varpi(t)\gr_x^2 U(\Xb_k) \rh)^2 $. Consequently,

\begin{eqnarray*}
        &  & \lf\| \gr_x^2 U(\Xb_k) \lf( (I_d +\varpi(t)\gr_x^2 U(\Xb_k))^{-1} -I_d \rh) \rh\|_2\\
        & \le & \lf\| \gr_x^2 U(\Xb_k) \rh\|_2 \cdot \lf(  \lf\| \varpi(t)\gr_x^2 U(\Xb_k) \rh\|_2 +  \lf\| \varpi(t)\gr_x^2 U(\Xb_k) \rh\|_2^2 \rh) \\
        & \le & 2 L^2 \D t_k^2,
\end{eqnarray*}
and
\begin{eqnarray*}
        &  & \lf\| \lf(e^{t-T_k} -1 \rh) \gr_x^2 U(\Xb_k) \lf(I_d +\varpi(t)\gr_x^2 U(\Xb_k)\rh)^{-1} \rh\|_2 \\
        & \le & 2\D t_k \lf\| \gr_x^2 U(\Xb_k) \rh\|_2 \cdot \lf\|\lf(I_d +\varpi(t)\gr_x^2 U(\Xb_k)\rh)^{-1}  \rh\|_2\\
        & \le & 2 L\D t_k \lf(\lf\|I_d\rh\|_2 + \lf\| \varpi(t)\gr_x^2 U(\Xb_k) \rh\|_2 + \lf\| \varpi(t)\gr_x^2 U(\Xb_k) \rh\|_2^2  \rh)\\
        & \le & 2 L \D t_k \lf(1 + L\D t_k^2 + L^2 \D t_k^4 \rh) \le 6L\D t_k.
\end{eqnarray*}
	As a result, $ \lf\|J_t(\Xb_k)\rh\|_2  \le  \sqrt{2} \max\lf\{ 2L^2\D t_k^2, 6L\D t_k \rh\} = 6\sqrt{2}L\D t_k$. This, together with \eqref{eq:boundJ2}, immediately leads to \eqref{eq:boundJ1}.
\end{proof}


%

	Recall that $\rho_k$ is the constant in the LSI satisfied by $\mu^*_k$ in Proposition~\ref{prop:LSI},
	and $\gamma(\eps_k)=\frac{4}{\eps_k}\lf( 1+ L^2 \rh) $.
	Denote  $c(\eps_k) :=\frac{1}{\gamma(\eps_k)} = \frac{\eps_k}{4\lf( 1+ L^2 \rh)}  $
	and
	$$
		\tilde{M}_k
		:=
		\begin{pmatrix}
 		\frac{7}{4} I_d & 2I_d -\gr_x^2 U\\
		2I_d - \gr_x^2 U & \lf( \frac{7}{4} + \frac{7}{8} \gamma(\eps_k)\eps_k  \rh)I_d - 2\gr_x^2 U
		\end{pmatrix} \in\M_{2d\times 2d}(\R).
	$$
	As the discrete time analogue to Lemma~\ref{Mpsd}, we have the following inequality for $\tilde{M}_k$. The proof is omitted because it is similar to that of Lemma~\ref{Mpsd}.

	\begin{lem} \label{disMpsd}
		Let $(\eps_k)_{k \ge 0}$ be given as in Assumption~\ref{assumption2} and $\gr_x U$ be $L$-Lipschitz.
		Then, for sufficiently large $k \ge 0$,
		we have $\tilde{M}_k \succeq c(\eps_k)\rho_k \lf( S  + \frac{\gamma(\eps_k)}{2\rho_k} I_{2d} \rh) $.
	\end{lem}

	We finally present the main result in this subsection.

	\begin{prop} \label{prop:42}
		Let Assumptions~\ref{assumption1}, \ref{assumption2} hold.
		Suppose in addition that $\gr_x^2 U$ is $L'$-Lipschitz and that $ \D t_k\le\min\lf\{ \frac{\eta^*}{2} , \frac{1}{L} \rh\} $ for sufficiently large $k \ge 0$.
		Then, there exist some constants $C_1 > 0$ and $C_2 > 0$ such that, for sufficiently large $k \ge 0$ and for all $\alpha > 0$,
		\begin{equation*}
			H_{\gamma(\eps_k)}(\mub_{k+1} | \mu^*_k) \le \lf( 1- C_1 \D t_k \; T_k^{-\lf(\frac{E_*}{E} +\alpha \rh)} \rh) H_{\gamma(\eps_k)}(\mub_k | \mu^*_k) + C_2 \D t_k^3 \; \lf(\log T_k\rh)^2.
		\end{equation*}
	\end{prop}
	\begin{proof}
	First, for sufficiently large $k\ge 0$, it follows from Lemma~\ref{disMpsd} that
	\begin{eqnarray*}
		&  &  -\;\E_{\mub_t} \lf[\lf\langle \gr_z \log\frac{\pb_t}{p^*_k}, \tilde{M}_k\gr_z\log\frac{\pb_t}{p^*_k} \rh\rangle \rh]\\
		& \le & -c(\eps_k) \rho_k \;\E_{\mub_t} \lf[\lf\langle \gr_z \log\frac{\pb_t}{p^*_k}, S\gr_z\log\frac{\pb_t}{p^*_k} \rh\rangle + \frac{\gamma(\eps_k)}{2\rho_k} \lf| \gr_z\log\frac{\pb_t}{p^*_k} \rh|^2 \rh].
	\end{eqnarray*}
	Using the LSI satisfied by $\mu^*_k$ in  Proposition~\ref{prop:LSI}, we obtain that
	\begin{equation} \label{eq:discrhoH}
		-\;\E_{\mub_t} \lf[\lf\langle \gr_z \log\frac{\pb_t}{p^*_k}, \tilde{M}_k\gr_z\log\frac{\pb_t}{p^*_k} \rh\rangle \rh] \le -c(\eps_k) \rho_k H_{\gamma(\eps_k)}(\mub_t|\mu^*_k).
	\end{equation}
	
	Next, combining \eqref{eq:dHk_inter} with Lemmas~\ref{lem:young} and \ref{lem:gradA}, we obtain that
	\begin{eqnarray*}
		\frac{\d}{\d t} H_{\gamma(\eps_k)}(\mub_t | \mu^*_k)
		& \le & -\;\E_{\mub_t} \lf[\lf\langle \gr_z \log\frac{\pb_t}{p^*_k}, \tilde{M}_k\gr_z\log\frac{\pb_t}{p^*_k} \rh\rangle \rh] + \frac{144d^{\frac{3}{2}}L^2}{\eps_k} \D t_k^2 \\
		&  & + \lf( \frac{(L')^2}{\eps_k} + \frac{4L^2}{\eps_k^2} \lf( 2+ \frac{1}{2}\eps_k\gamma(\eps_k) \rh) \rh) \;\E \lf[\lf|\Xb_t-\Xb_k\rh|^2\rh].
	\end{eqnarray*}
	Plugging \eqref{eq:discrhoH} into the above inequality and recalling Proposition~\ref{prop:disboundx}, we derive
	\begin{equation*}
		\frac{\d}{\d t} H_{\gamma(\eps_k)}(\mub_t | \mu^*_k)
		\le -c(\eps_k) \rho_k H_{\gamma(\eps_k)}(\mub_t | \mu^*_k) + C \eps_k^{-2} \D t_k^2.
	\end{equation*}
	Recall that $\eps_k \; =\; \frac{E}{\log T_k}$ with $E>E_*$ for sufficiently large $k \ge 0$ and that $\rho_k = \chi(\eps_k)e^{-\frac{E_*}{\eps_k}} $, where  $\eps_k\log\chi(\eps_k) \longrightarrow 0 $ as $k \longrightarrow \infty$.
	Then, for sufficiently large $k \ge 0$ and for all $\alpha>0$, we have
	\begin{equation*}
		\frac{\d}{\d t} H_{\gamma(\eps_k)}(\mub_t | \mu^*_k) \le -C T_k^{-\lf(\frac{E_*}{E} + \alpha \rh)} H_{\gamma(\eps_k)}(\mub_t | \mu^*_k) + C \D t_k^2 (\log T_k)^2.
	\end{equation*}
	By computing $\frac{\d}{\d t}\lf(e^{C\; t\; T_k^{-\lf(\frac{E_*}{E} +\alpha \rh)}} H_{\gamma(\eps_k)}(\mub_t | \mu^*_k)  \rh)  $ and then integrating it  from $T_{k}$ to $T_{k+1} $, we obtain
	\begin{eqnarray*}
		&  & H_{\gamma(\eps_k)}(\mub_{k+1} | \mu^*_k)\\
		& \le & e^{- C\D t_k T_k^{-\lf(\frac{E_*}{E} +\alpha \rh)}} H_{\gamma(\eps_k)}(\mub_k | \mu^*_k) + C\D t_k^2 (\log T_k)^2\; T_k^{\lf( \frac{E_*}{E} +\alpha \rh)}
			\Big(1- e^{- C\D t_k T_k^{-\lf(\frac{E_*}{E} +\alpha \rh)}}  \Big)\\
		& \le & \Big( 1- C_1 \D t_k T_k^{-\lf(\frac{E_*}{E} +\alpha \rh)}  \Big) H_{\gamma(\eps_k)}(\mub_k | \mu^*_k) + C_2\D t_k^3\; (\log T_k)^2,
	\end{eqnarray*}
	where $C_1 > 0$ and $C_2 > 0$ are some constants independent of $k$. The proof then completes.
\end{proof}

\begin{rmk}
An alternative discrete time scheme of \eqref{eq:kl} would be the Euler-Maruyama scheme:
\begin{equation}
       \label{eq:Euler}
\begin{cases}
  \Xb_{k+1}^E & = \Xb_k^E + \Yb_k^E \D t_k \\
  \Yb_{k+1}^E & = \Yb_k^E - \gr_x U(\Xb_k^E)\D t_k - \Yb_k^E \D t_k + \sqrt{2\eps_k\D t_k}\; W_k,
\end{cases}
\end{equation}
where $(W_k\; ; \; k=0,1,2,\dots)$ are i.i.d standard normal vectors in $\R^d$.
In order to compute the one-step distorted entropy dissipation of Euler-Maruyama scheme \eqref{eq:Euler} with fixed temperature $\eps_k$, conditional on $(\Xb_k^E, \Yb^E_k)$, we need to couple \eqref{eq:Euler} with a continuous time process in every time interval $[T_k, T_{k+1}]$:
\begin{equation}
       \label{eq:Eulercon}
\begin{cases}
  \d \Xb_t^E & =  \Yb_k^E \d t \\
  \d \Yb_t^E & = - \gr_x U(\Xb_k^E) \d t - \Yb_k^E \d t + \sqrt{2\eps_k} \d B_t,
\end{cases}
\end{equation}
where $B_t$ is a Brownian motion independent of $(\Xb_k^E, \Yb^E_k)$. Denote the law of $(\Xb_t^E ,\Yb_t^E)$ as $\mub^E_t$ with density function $\pb^E_t$ and denote $\th_t^E\; := \; \sqrt{\pb_t^E/ p^*_k} $. Similar to the computation of $\frac{\d}{\d t} H_{\gamma(\eps_k)} \lf(\mub_t | \mu^*_k\rh)$ for the second-order scheme, we obtain, for the Euler-Maruyama scheme, that
\begin{eqnarray*}
    &  & \frac{\d}{\d t} H_{\gamma(\eps_k)}\lf(\mub^E_t | \mu^*_k\rh)\\
    & = & -4\;\E_{\mub_t^E} \lf[\lf\langle \gr_z \log\th_t^E, M_k\gr_z\log\th_t^E \rh\rangle \rh]  -8\eps_k \;\E_{\mub_t^E} \lf[\lf\langle \gr_z\gr_y\log\th_t^E, S\gr_z\gr_y\log\th_t^E  \rh\rangle_F  \rh] \\
        &  &  + 4\;\E_{\mub_t} \lf[\lf\langle \gr_z\gr_y\log\th_t^E, S\gr_z A_t^E(z) \rh\rangle_F  \rh] \\
        && + 4\;\E_{\mub_t^E} \lf[\lf\langle \gr_z\gr_y\log\th_t^E -\gr_z\gr_x\log\th_t^E, S\gr_z B_t^E (z) \rh\rangle_F  \rh]   \\
        & & + \int \Big\langle \frac{4}{\eps_k}\gr_x\log \th_t^E + \lf(2\gamma(\eps_k) +\frac{4}{\eps_k}  \rh)\gr_y\log\th_t^E, \\
        && \;\;\;\;\;  \;\;\;\;\;  \;\;\;\;\; \int (\gr U(x) - \gr U(x_k) ) \pb^E_{t|T_k}(z|z_k) \pb^E_k(z_k) \d z_k  \Big\rangle\d z\\
        &  & + \frac{4}{\eps_k}\int \Big\langle \lf(I_d - \gr_x^2 U(x) \rh) (\gr_x\log\th_t^E + \gr_y\log\th_t^E ),\\
        && \;\;\;\;\;  \;\;\;\;\;  \;\;\;\;\;\; \;\;\;\;\;  \;\;\;\;\;  \;\;\;\;\;  \int (y_t - y_k) \pb^E_{t|T_k}(z|z_k) \pb^E_k(z_k) \d z_k \Big\rangle \d z \\
        &  &  +2\gamma(\eps_k) \int \lf\langle -\gr_x\log\th_t^E +\gr_y\log\th_t^E, \; \int (y_t - y_k) \pb^E_{t|T_k}(z|z_k) \pb^E_k(z_k) \d z_k  \rh\rangle \d z,
\end{eqnarray*}
where
\begin{equation*}
    \pb^E_{t|T_k}(z|z_k)\; := \; p(\Zb_t^E=z\; | \; \Zb_k^E =z_k),
\end{equation*}
\begin{equation*}
    A_t^E(z)\; := \; \frac{1}{\pb_t^E(z)} \int \lf( \gr_x U(x) - \gr_x U (x_k) \rh) \pb^E_{t|T_k} (z|z_k) \pb^E_k(z_k) \d z_k,
\end{equation*}
\begin{equation*}
    B_t^E(z)\; := \; \frac{1}{\pb_t^E(z)} \int \lf( y - y_k \rh) \pb^E_{t|T_k} (z|z_k) \pb^E_k(z_k) \d z_k.
\end{equation*}

Similar to the analysis of the second-order scheme, we set $\D t_k\le\frac{1}{2\sqrt{L}}$ with $L\ge 1$ and $\varpi^E(t) = \frac{(t-T_k)^2}{1-(t-T_k)}  $. Then, we have
\begin{eqnarray*}
    \gr_z A_t^E(z) = \;\int \lf( \begin{array}{c}
        \gr_x^2 U(x) - \gr_x^2 U(x_k) - \gr_x^2 U(x_k)\lf( (I_d +\varpi^E(t)\gr_x^2 U(x_k))^{-1} -I_d \rh)   \\
        \frac{t-T_k}{1-(t-T_k)} \gr_x^2 U(x_k) (I_d +\varpi^E(t)\gr_x^2 U(x_k))^{-1}
        \end{array}
\rh) \\
\pb^E_{T_k |t} (z_k | z) \d z_k,
\end{eqnarray*}
\begin{eqnarray*}
    \gr_z B_t^E(z) = \int \lf( \begin{array}{c}
        - (t-T_k)^2 \lf( \lf( 1- (t-T_k) \rh) I_d + (t-T_k)^2 \gr_x^2 U(x_k) \rh)^{-1}\gr_x^2 U(x_k)   \\
        I_d - \lf( \lf( 1- (t-T_k) \rh) I_d + (t-T_k)^2 \gr_x^2 U(x_k) \rh)^{-1}
        \end{array}
\rh) \\ \pb^E_{T_k |t} (z_k | z) \d z_k.
\end{eqnarray*}

However, compared with $\frac{\d}{\d t} H_{\gamma(\eps_k)}\lf(\mub_t | \mu^*_k  \rh) $ for the second-order scheme, the entropy dissipation we obtained the above for Euler-Maruyama scheme have some extra terms.
	In particular, it is still unclear to us how to handle the term
\begin{equation*}
    -4\; \E_{\mub_t^E} \lf[\lf\langle \gr_z\gr_x\log \th_t^E, \; S\gr_z B_t^E(z)   \rh\rangle_F \rh].
\end{equation*}

\end{rmk}


\subsection{One-step entropy evolution with cooling temperature}
	Next, we compute the term in \eqref{eq:derdis2}, which is the one-step change of $H_{\gamma(\eps)}\lf( \mub_{k+1} | \mu^*_{\eps} \rh)$ when the temperature $\eps$ in the discrete scheme \eqref{eq:diskl} is updated from $\eps_{k}$ to $\eps_{k+1}$ in the $k$-th step.

	\begin{lem} \label{disepsilon}
		Let Assumptions~\ref{assumption1} and \ref{assumption2} hold.
		Then, there exists a constant $C > 0$ such that for sufficiently large $k \ge 0$ and for all $\alpha > 0$,
		\begin{equation*}
		H_{\gamma(\eps_{k+1})}(\mub_{k+1}|\mu^*_{k+1}) \le H_{\gamma(\eps_k)}(\mub_{k+1}|\mu^*_k) + C \D t_k T_k^{-1+\alpha}.
		\end{equation*}
	\end{lem}
	\begin{proof}
	The proof is similar to that of  Lemma~\ref{conepsilon}. We use $C > 0$ to denote a generic constant whose value may change from line to line.
	First, use similar arguments to the ones in the proof of Lemma~\ref{conepsilon}, we obtain
	\begin{eqnarray*}
		&  & \dr_{\eps} \;\E_{\mub_{k+1}} \lf[ \lf\langle \gr_z\log\frac{\pb_{k+1}}{p^*_{\eps}}, S\,\gr_z\log \frac{\pb_{k+1}}{p^*_{\eps}}\rh\rangle \rh]\\
		& \le & \;\E_{\mub_{k+1}} \lf[ \lf\langle \gr_z \log \frac{\pb_{k+1}}{p^*_{\eps}}, S\,\gr_z \log \frac{\pb_{k+1}}{p^*_{\eps}} \rh\rangle \rh]
		+ C \eps^{-4} (1+ \Cb_0)
	\end{eqnarray*}
	and
	\begin{equation*}
		\dr_{\eps} \lf( \gamma(\eps) \;\E_{\mub_{k+1}} \lf[ \log\frac{\pb_{k+1}}{p^*_{\eps}}  \rh] \rh) \le C \eps^{-2} \;\E_{\mub_{k+1}} \lf[ \log\frac{\pb_{k+1}}{p^*_{\eps}}  \rh]
		\; + \;
		C \eps^{-3} (1+\Cb_0) .
	\end{equation*}
	It follows that
	\begin{eqnarray*}
		\dr_{\eps} H_{\gamma(\eps)}\lf(\mub_{k+1}|\mu^*_{\eps}\rh)
		& \le & (1+\eps^{-1}) H_{\gamma(\eps)}(\mub_{k+1}|\mu^*_{\eps}) + C \eps^{-4}\; (1 + \Cb_0).
	\end{eqnarray*}
	Recall that $\eps_t = \frac{E}{\log t}$ with $E>E_*$ for large $t$. 
	Then, for any sufficiently large $k$ and $t\in[T_k, T_{k+1}] $, we have $\eps_t\in[\eps_{k+1}, \eps_k] $ and
	\begin{equation*}
		\dr_{\eps} H_{\gamma(\eps_t)}(\mub_{k+1}|\mu^*_{\eps_t}) \le C\log T_{k+1}\; H_{\gamma(\eps_t)}(\mub_{k+1}|\mu^*_{\eps_t})
		\; + \;
		C(\log T_{k+1} )^4\, (1 + \Cb_0) .
	\end{equation*}
	It follows by Gr\"{o}nwall lemma that
	\begin{eqnarray*}
		H_{\gamma(\eps_{k+1})}(\mub_{k+1}|\mu^*_{k+1})
		\le
		 e^{C (\eps_{k+1} -\eps_k )\log T_{k+1} } H_{\gamma(\eps_k)}(\mub_{k+1}|\mu^*_k) + C (\eps_k - \eps_{k+1} ) (\log T_{k+1})^4.
	\end{eqnarray*}
	Notice that
	$e^{C (\eps_{k+1} -\eps_k )\log T_{k+1} } \le 1 $ and $\eps_{k} -\eps_{k+1}\le C\D t_k T_k^{-1}(\log T_k)^{-2} $. As a result, for any $\alpha>0$, we have
\begin{eqnarray*}
        H_{\gamma(\eps_{k+1})}(\mub_{k+1}|\mu^*_{k+1}) & \le & H_{\gamma(\eps_k)}(\mub_{k+1}|\mu^*_k) + C \D t_k T_k^{-1} (\log T_{k+1})^4 (\log T_k)^{-2}\\
    & \le & H_{\gamma(\eps_k)}(\mub_{k+1}|\mu^*_k) + C_3 \D t_k T_k^{-1+\alpha}.
\end{eqnarray*}

\end{proof}


\subsection{Proof of Theorem~\ref{thm2}}

	For each $k \ge 0$, let $\tilde{Z}_k=(\tilde{X}_k , \tilde{Y}_k)$ be a random variable with distribution $\mu^*_k$ in the same probability space supporting $(\Xb, \Yb)$.
	Similar to \eqref{bound}, we can show that, for each $ \delta>0$,
	\begin{equation} \label{disbound}
		\P \lf(U(\Xb_k)>\delta \rh)
		~\le~
		\P \lf(U(\tilde{X}_k)>\delta \rh)
		~+~
		\sqrt{2H_{\gamma(\eps_k)}(\mub_k | \mu^*_k)}.
	\end{equation}
	The term $\P \lf(U(\tilde{X}_k)>\delta \rh)$ can be handled by Lemma~\ref{Laplace} to obtain the desired convergence rate.
	For the convergence of $\sqrt{2H_{\gamma(\eps_k)}(\mub_k | \mu^*_k)}$, we can apply  Proposition~\ref{prop:42} and Lemma~\ref{disepsilon} to show that 
	there are some positive constants $c_1$, $c_1'$, $c_2$, and $k_0 \ge 0$ such that
	\begin{eqnarray*}
		&  & H_{\gamma(\eps_{k+1})}(\mub_{k+1}|\mu^*_{k+1}) - H_{\gamma(\eps_k)}(\mub_k|\mu^*_k)\\
		& \le &  -c_1 \D t_k T_k^{-\lf(\frac{E_*}{E} +\alpha \rh)} H_{\gamma(\eps_k)}(\mub_k|\mu^*_k) + c_1' \D t_k^3(\log T_k)^2 + c_2 \D t_k T_k^{-1+\alpha},
	\end{eqnarray*}
	for all $k \ge k_0$.
	Because $\limsup_{k\to\infty} \D t_k\sqrt{T_k} < \infty $, we can find some positive constant $c_3$ such that
\begin{equation*}
    H_{\gamma(\eps_{k+1})}(\mub_{k+1}|\mu^*_{k+1}) \le \lf( 1 - c_1 \D t_k T_k^{-\lf(\frac{E_*}{E} +\alpha \rh)}\rh) H_{\gamma(\eps_k)}(\mub_k|\mu^*_k) + c_3 \D t_k T_k^{-1+\alpha}.
\end{equation*}
It is sufficient to consider $\alpha<\frac{1}{2}\lf(1-\frac{E_*}{E}  \rh)  $. Then,
\begin{eqnarray*}
&  & H_{\gamma(\eps_{k+1})}(\mub_{k+1}|\mu^*_{k+1}) -\frac{2c_3}{c_1} T_{k+1}^{-\lf( 1-\frac{E_*}{E}-2\alpha \rh) }\\
& \le & \lf( 1 - c_1 \D t_k T_k^{-\lf(\frac{E_*}{E} +\alpha \rh)} \rh) \lf(H_{\gamma(\eps_k)}(\mub_k|\mu^*_k) -\frac{2c_3}{c_1} T_k^{-\lf( 1-\frac{E_*}{E}-2\alpha \rh) }  \rh),
\end{eqnarray*}
	which implies that, for $k \ge k_0$,
	\begin{eqnarray*}
		&&
		H_{\gamma(\eps_k)}(\mub_k|\mu^*_k)
		~\le~
		 \frac{2c_3}{c_1} T_k^{-\lf( 1-\frac{E_*}{E}-2\alpha \rh) } \\
        		&&\;\;\;\;\;
		 + \prod_{j=k_0}^{k-1} \lf( 1 - c_1 \;\D t_j T_j^{-\lf(\frac{E_*}{E} +\alpha \rh)} \rh) \cdot \lf(H_{\gamma(\eps_{k_0})}(\mub_{k_0}|\mu^*_{k_0}) -\frac{2c_3}{c_1} T_{k_0}^{-\lf( 1-\frac{E_*}{E}-2\alpha \rh) }  \rh).
	\end{eqnarray*}
	Further, notice that $T_k \longrightarrow \infty$ and $\frac{E_*}{E} +\alpha<1$.
	Then,
	$$
		\sum_{i =0}^{\infty} \;\D t_i T_i^{-\lf(\frac{E_*}{E} +\alpha \rh)} = \infty ,
		~\mbox{and hence}~
		\prod_{j=k_0}^{\infty} \lf( 1 - c_1 \;\Delta t_j T_j^{-\lf(\frac{E_*}{E} +\alpha \rh)} \rh) =  0.
	$$
	Finally, by Lemma~\ref{disfiniteH} which will be presented and shown momentarily, we know that $H_{\gamma(\eps_{k_0})} (\mub_{k_0} | \mu^*_{k_0} )$ is finite.
	Then, there exists a positive constant $C > 0$ such that
	\begin{equation*}
			H_{\gamma(\eps_k)}(\mub_k|\mu^*_k) \le C T_k^{-\lf( 1-\frac{E_*}{E}-2\alpha \rh)}.
	\end{equation*}
	This completes the proof.
	\qed

	We finally provide a discrete time analogue of the result in Lemma~\ref{finiteH}, which is needed in the proof of Theorem~\ref{thm2}.

	\begin{lem} \label{disfiniteH}
		For every $k \ge 0$, the distorted entropy $H_{\gamma(\eps_k)}\lf( \mub_k | \mu^*_k \rh) $ is finite.
	\end{lem}

	\begin{proof}
		As in the proof of Lemma~\ref{finiteH}, it suffices to prove that the Fisher information $I(\mub_k|\mu^*_k)$ is finite for every finite $k \ge 0$.
		Again, we use $C$ to denote a generic positive constant whose value may change from line to line.

	For $t\in[T_k,T_{k+1}]$, with $\th_t=\sqrt{\frac{\pb_t}{p^*_k}}$, recall that
	\begin{equation*}
		I(\mub_t | \mu^*_k) = \E_{\mub_t} \lf[ \lf| \gr_z \log\frac{\pb_t}{p_k^*} \rh|^2 \rh] = 4\;\E_{\mub_t}\lf[\lf| \gr_z \log\th_t \rh|^2  \rh].
	\end{equation*}
	Using a similar computation to the one in the calculation of $H_{\gamma(\eps_k)}\lf(\mub_t | \mu^*_k  \rh)$ in Section~\ref{s:disdissipation}, we can obtain that
	\begin{eqnarray*}
		\frac{\d}{\d t} I(\mub_t | \mu^*_k)  & = &  \int \lf<\gr_z \lf( \frac{\delta I(\mub_t | \mu^*_k)}{\delta\; \mub_t} \rh),\; v_{t,k}  \rh>\; \pb_t(z)\;\d z \\
		&  & + \int \lf<\gr_z \lf( \frac{\delta I(\mub_t | \mu^*_k)}{\delta\; \mub_t} \rh),\; \int \pb_{t|T_k}(z|z_k) (\hat{v}_{t,k}- v_{t,k} ) \pb_k(z_k) \d z_k  \rh>\; \d z,
	\end{eqnarray*}
	where
	\begin{equation*}
		\frac{\delta I(\mub_t | \mu^*_k)}{\delta\; \mub_t} = \frac{4}{\th_t} \tgr_z^*\gr_z\th_t.
	\end{equation*}
	Following the same steps as in \eqref{ex1}, \eqref{ex2} and \eqref{ex3}, we further obtain that
	\begin{eqnarray} \label{eq:I1}
		&  & \int \lf<\gr_z \lf( \frac{\delta I(\mub_t | \mu^*_k)}{\delta\; \mub_t} \rh),\; v_{t,k}  \rh>\; \pb_t(z)\;\d z\\
		& = & -4 \;\E_{\mub_t} \lf[\lf\langle \gr_z \log \th_t, M_I \gr_z \log \th_t \rh\rangle  \rh] \no  -8\eps_k  \;\E_{\mub_t} \lf[ \lf\|
		\gr_z\gr_y \log \th_t \rh\|_F \rh]\no,
	\end{eqnarray}
	where $M_I=\begin{pmatrix}
 		0 & I_d -\gr_x^2 U\\
		 I_d - \gr_x^2 U & 2I_d
		\end{pmatrix}$.
	Further, by similar arguments to the ones in Lemma~\ref{lemb2}, and with $A_t(z)$ as given by \eqref{eq:A}, we have
	\begin{eqnarray} \label{eq:I2}
		&  & \int \lf<\gr_z \lf( \frac{\delta I(\mub_t | \mu^*_k)}{\delta\; \mub_t} \rh),\; \int \pb_{t|T_k}(z|z_k) (\hat{v}_{t,k}- v_{t,k} ) \pb_k(z_k) \d z_k  \rh>\; \d z\\
		& = &  \frac{4}{\eps_k} \int\bigg\langle \gr_y\log\th_t, \; \int \lf(\gr_x U(x) -\gr_x U(x_k) \rh)\; \pb_{t|T_k}(z|z_k) \; \pb_k(z_k)\d z_k \bigg\rangle \d z \no \\
		&  & + 4\int \lf\langle \gr_z\gr_y\log\th_t, \gr_z A_t(z) \rh\rangle_F \d\mub_t (z).\no
	\end{eqnarray}
	Following the same arguments as in Lemma~\ref{lem:young} and \ref{lem:gradA}, we can further obtain the inequalities
	\begin{eqnarray*}
		 &  & \frac{4}{\eps_k} \int\bigg\langle \gr_y\log\th_t, \; \int \lf(\gr_x U(x) -\gr_x U(x_k) \rh)\; \pb_{t|T_k}(z|z_k) \; \pb_k(z_k)\d z_k \bigg\rangle \d z\\
		 & \le & \E_{\mub_t} \lf[\lf\| \gr_y\log\th_t \rh\|^2  \rh] + \frac{4L^2}{\eps_k^2} \E\lf[\lf|\Xb_t- \Xb_k  \rh|^2 \rh],
	\end{eqnarray*}
	and
	\begin{eqnarray*}
		 &  & 4\int \lf\langle \gr_z\gr_y\log\th_t, \gr_z A_t(z) \rh\rangle_F \d\mub_t (z)\\
		 & \le & 8\eps_k\;\E_{\mub_t} \lf[ \lf\| \gr_z\gr_y \log\th_t \rh\|^2_F \rh]  + \frac{(L')^2}{\eps_k} \;\E \lf[\lf|\Xb_t-\Xb_k\rh|^2\rh] + \frac{72}{\eps_k} d L^2 \D t_k^2.
	\end{eqnarray*}

	Next, by Proposition~\ref{prop:disboundx}, \eqref{eq:I1} and \eqref{eq:I2}, it follows that
\begin{equation*}
    \frac{\d}{\d t} I(\mub_t | \mu^*_k)  \le  - \; \E_{\mub_t} \lf[ \lf\langle \gr_z\log\frac{\pb_t}{p^*_k}, \; \tilde{M}_I \gr_z \log \frac{\pb_t}{p^*_k}  \rh\rangle  \rh] + C \eps_k^{-2} \D t_k^2,
\end{equation*}
where $\tilde{M}_I=\begin{pmatrix}
 0 & I_d -\gr_x^2 U\\
 I_d - \gr_x^2 U & \frac{7}{4}I_d
\end{pmatrix} \succeq -(1+L)I_{2d}$ due to the $L$-Lipschitz of $\gr_x U$.
	Then,
	\begin{equation*}
		\frac{\d}{\d t} I(\mub_t | \mu^*_k)  \le (1+L) I(\mub_t | \mu^*_k) + C \eps_k^{-2} \D t_k^2,
	\end{equation*}
	which implies that
	\begin{equation} \label{eq:I3}
		I(\mub_{k+1} | \mu^*_k) \le e^{(1+L)\D t_k} I(\mub_k | \mu^*_k) + C\D t_k^3.
	\end{equation}
	
	Finally, following similar computation steps to the ones in Lemma~\ref{disepsilon}, we obtain that, for all $\eps > 0$,
	\begin{eqnarray*}
		\dr_{\eps} I\lf(\mub_{k+1} | \mu^*_{\eps}  \rh)
		~=~
		\dr_{\eps} \E_{\mub_{k+1}} \lf[ \lf| \gr_z\log\frac{\pb_{k+1}}{p^*_{\eps}}  \rh|^2  \rh]
		~\le~
		 I (\mub_{k+1} | \mu^*_{\eps}) +C\eps^{-4}.
	\end{eqnarray*}
	The above, together with \eqref{eq:I3}, implies that
	\begin{eqnarray*}
		I(\mub_{k+1} | \mu^*_{k+1}) & \le & e^{(\eps_{k+1} - \eps_k)} I(\mub_{k+1} | \mu^*_k) + C e^{\eps_{k+1}} \int_{\eps_k}^{\eps_{k+1}} \eps^{-4} e^{-\eps} \d \eps\\
		& \le & e^{(1+L)\D t_k} I(\mub_k | \mu^*_k) + C\D t_k^3\\
		& \le & e^{(1+L)\T_{k+1}} I(\mu_0 | \mu^*_{\eps_0}) + e^{(1+L)\T_{k+1}} \sum_{i=0}^{k} \D t_i^3,
	\end{eqnarray*}
	where the r.h.s. of the last inequality in the above is clearly finite for every $k \ge 0$.
\end{proof}

%
%
%
%
%
%
%
%
%

\bibliographystyle{abbrv}
\bibliography{LongTitle,reference}

\end{document}